 \documentclass[dvips,arXiv]{degruyter-journal_2} 

\usepackage[latin1]{inputenc}
\usepackage{colortbl}
\usepackage{mathrsfs}\usepackage{amscd}\usepackage{graphicx}
\usepackage{subfigure}\usepackage{amsfonts}\usepackage{amsxtra}
\usepackage{color}
\usepackage{multirow}

\DeclareMathOperator{\range}{range}
\DeclareMathOperator{\sign}{sign}

\DeclareMathOperator{\id}{id}

\newcommand{\R}{\mathbb{R}}\newcommand{\C}{\mathbb{C}}\newcommand{\N}{\mathbb{N}}\newcommand{\abs}[1]{\lvert #1\rvert}
\renewcommand{\l}{\ell}
\renewcommand{\thefootnote}{\fnsymbol{footnote}}

\title{Shrinkage Rules for Variational Minimization Problems and Applications to Analytical Ultracentrifugation}
\headlinetitle{Shrinkage Rules for Variational Minimization Problems}

\lastnameone{Ehler}
\firstnameone{Martin}
\nameshortone{M.~Ehler}
\addressone{Section on Medical Biophysics, Eunice Kennedy Shriver National Institute of Child Health and Human Development, National Institutes of Health, 
9 Memorial Drive, Bethesda, MD 20892 \\ and \\ Norbert Wiener Center, Department of Mathematics, University of Maryland, College Park, MD 20742}
\countryone{USA}
\emailone{ehlermar@mail.nih.gov, ehlermar@math.umd.edu}

\lastnametwo{}
\firstnametwo{}
\nameshorttwo{}
\addresstwo{}
\countrytwo{}
\emailtwo{}

\lastnamethree{}
\firstnamethree{}
\nameshortthree{}
\addressthree{}
\countrythree{}
\emailthree{}

\lastnamefour{}
\firstnamefour{}
\nameshortfour{}
\addressfour{}
\countryfour{}
\emailfour{}

\lastnamefive{}
\firstnamefive{}
\nameshortfive{}
\addressfive{}
\countryfive{}
\emailfive{}

\abstract{Finding a sparse representation of a possibly noisy signal can be modeled as a variational minimization with $\ell_q$-sparsity constraints for $q$ less than one. Especially for real-time, on-line, or iterative applications, in which problems of this type have to be solved multiple times, one needs fast algorithms to compute these minimizers. 

Identifying the exact minimizers is computationally expensive. We consider minimization up to a constant factor to circumvent this limitation. We verify that $q$-dependent modifications of shrinkage rules provide closed formulas for such minimizers. Therefore, their computation is extremely fast. We also introduce a new shrinkage rule which is adapted to $q$. 

To support the theoretical results, the proposed method is applied to Landweber iteration with shrinkage used at each iteration step. This approach is utilized to solve the ill-posed problem of analytic ultracentrifugation, a method to determine the size distribution of macromolecules. For relatively pure solutes, our proposed scheme leads to sparser solutions with sharper peaks, higher resolution, and smaller residuals than standard regularization for this problem.}



\researchsupported{The research was funded by the Intramural Research Program of the National Institutes of Child Health and Human Development and by the Research Career Transition Awards Program 575910 of the National Institutes of Health and the German Science Foundation.}

\acknowledgments{The author is grateful to Dr.~Peter Schuck for providing the analytical ultracentrifugation data and for many insightful discussions.}

\begin{document}

\section{Introduction}

Decomposing signals into simple building blocks and reconstructing from shrinked coefficients are used in signal representation and processing, e.g., wavelet shrinkage is applied to noise and clutter reduction in speckled SAR images, 
cf.~\cite{coastlines}. 
Statistical approaches and Bayesian objectives for noise removal make use of various shrinkage strategies, cf.~\cite{donoho_1,Elad,Jeff
,surerisk}. 
Variational models as in \cite{variational_denoising} justify shrinkage by smoothness estimates of the unperturbed signal. Other shrinkage rules are derived from a diffusion approach in \cite{WeickertSteidl}. 

Signal approximation with sparsity constraints leads to variational minimization problems, and the denoising approach in \cite{variational_denoising} is a particular case. The expression to be minimized is a sum of an approximation error and a penalty term which involves weighted $\ell_q$-constraints, see Section \ref{subsection Variational Minimization Problems}. In \cite{daub_inverse_defrise}, iteratively shrinking coefficients of an orthonormal basis expansion provides a sequence converging towards the minimizer. The method covers the convex case $q\in [1,2]$, but sparse signal representation, coding, signal analysis, and the treatment of operator equations require the consideration of redundant basis-like systems and the nonconvex case $q\in[0,1)$ as well, see for instance \cite{BalanRick,Chartrand,DahlkeMassimoRaasch,Dahlke:2007ab,ehler4,FornasierRauhut,Stevenson:2008aa} and references therein. By using hard-shrinkage, the algorithm in  
 \cite{daub_inverse_defrise} converges towards a local minimum for $q=0$, cf.~\cite{Blumensath1}. Under the restricted isometry property (RIP) \cite{CanRomTao,DonElaTem}, the iteration converges towards the exact minimum, cf.~\cite{Blumensath:2009aa}. However, RIP does not hold in many situations and therefore the local minimum could be far off the global minimum. The approach does not cover $q\in(0,1)$, and, for applications where computation time is crucial, a faster algorithm is desirable.

In the present paper, we obtain complementary results for $q\in[0,1)$ in terms of minimization up to a constant factor. In fact, we verify that such a minimization can be derived from $q$-dependent modifications of shrinkage rules. This means we have a closed formula for these minimizers, which allows for a fast computation. We also introduce new shrinkage rules which are adapted to $q$. 
We then propose a Landweber iteration with these new shrinkage strategies applied in each step to treat sparsity constraints for $q\in (0,1)$, cf.~\cite{Blumensath1,daub_inverse_defrise} for soft- and hard-shrinkage. This approach is then applied to the ill-posed problem of sedimentation velocity analytical ultracentrifugation, a method to determine the size distribution of interacting macromolecules \cite{Cox:1969aa,Schuck:2000aa}. Its physical model leads to a Fredholm integral equation that needs to be regularized. For highly pure monomers, the solution is expected to be highly sparse with few sharp peaks. Our numerical experiments suggest that our proposed iterative shrinkage scheme leads to sharper peaks, fewer nonzero entries, and smaller residuals. Thus, it provides a useful add-on to standard analytical ultracentrifugation analysis, cf.~\cite{Brown:2008aa,Schuck:2000aa}.   


The outline is as follows: In Section \ref{section:frames problems}, we present the variational problems under consideration and we recall the concept of frames. We introduce shrinkage rules in Section \ref{shrinkage rules}. The main results about minimization up to a constant factor are presented in Section \ref{section main results}, and in Section \ref{section:sparse approximation} we apply the results to sparse signal representation. We introduce a new family of shrinkage rules in Section \ref{section interpolation}. The modified Landweber iteration is explicitly introduced in Section \ref{section:Landweber}, where we also present numerical results about sedimentation velocity analytical ultracentrifugation. Conclusions are given in Section \ref{section:conclusions}.

\section{Variational Problems and Frames}\label{section:frames problems}

\subsection{Variational Minimization Problems}\label{subsection Variational Minimization Problems}
Let $L$ be a bounded operator between two Hilbert spaces $\mathcal{H}$ and $\mathcal{H}'$, and let $\{\tilde{f}_n\}_{n\in\mathcal{N}}$ be a countable collection in $\mathcal{H}$. Given $h\in\mathcal{H}'$, we consider the minimization problem
\begin{equation}\label{problem 1}
\min_{g\in\mathcal{H}}\big( \|h-Lg\|^2_{\mathcal{H}'} + \sum_{n\in\mathcal{N}} \alpha_n |\langle g,\tilde{f}_n\rangle|^q\big),
\end{equation}
where $q\in (0,2]$, $(\alpha_n)_{n\in\mathcal{N}}$ is a sequence of nonnegative numbers, and $\langle\cdot,\cdot\rangle$ denotes the inner product on $\mathcal{H}$. This makes also sense for $q=0$ with the convention $0^0=0$, and the penalty term then counts the nonzero entries of $(\langle g,\tilde{f}_n\rangle)_{n\in\mathcal{N}}$ weighted by $(\alpha_n)_{n\in\mathcal{N}}$. For $\mathcal{H}=\mathcal{H}'$ and $L=\id_\mathcal{H}$, problem \eqref{problem 1} is relevant in wavelet based signal denoising. There, $\{\tilde{f}_n\}_{n\in\mathcal{N}}$ is a wavelet system, and the sparsity constraint on the right hand side of \eqref{problem 1} is related to the Besov regularity of the signal to be recovered, see \cite{variational_denoising} for details. Our approach is neither restricted to $L$ being the identity nor must $L$ be injective. However, we assume that it 
has a  bounded pseudo inverse, i.e.~there is a bounded operator $L^\#:\mathcal{H}'\mapsto\mathcal{H}$ such that $LL^\#L=L$. Thus, we first consider well-posed problems and are therefore more restrictive than in \cite{daub_inverse_defrise}. Nevertheless, we address ill-posed problems in Section \ref{section:Landweber} by extending the iterative shrinkage procedure introduced in \cite{daub_inverse_defrise}. 

The sequence $(\alpha_n)_{n\in\mathcal{N}}$ is a collection of variable parameters which must be fitted to $h$ and $L$. If all components of $\alpha_n=\alpha$ are identical, then  
\begin{equation}\label{eq:curve}
\alpha \mapsto \big(\|h-Lg^{\alpha}\|^2_{\mathcal{H}'} , \sum_{n\in\mathcal{N}} |\langle g^\alpha,\tilde{f}_n\rangle|^q  \big)
\end{equation}
is considered as a curve in $\R^2$, where $g^\alpha$ is a minimizer of \eqref{problem 1}, and one finally chooses $\alpha$ according to a point of maximal curvature, see \cite{Hansen} and \cite{curvature} for the \emph{L-curve} and \emph{H-curve criterion}, respectively. It requires to compute minimizers $g^\alpha$ for many different values of $\alpha$, and $g^\alpha$ must be efficiently computable. This is another motivation for avoiding costly iterative minimization schemes beside real-time and on-line applications.

Handling nonstationary noise requires that $(\alpha_n)_{n\in\mathcal{N}}$ depends on $n$, but it is often still reasonable to assume that there are positive constants $a$ and $b$ such that 
\begin{equation}\label{eq:boundedness condition on delta}
a\leq \alpha_n\leq b, \text{ for all $n\in\mathcal{N}$.}
\end{equation}

\subsection{Bi-frames}\label{subsection Wavelet Bi-Frames}
The singular value decomposition of $L$ is considered in \cite{Lorenz:2008ab} to address $q\in[0,1)$. The system  $\{\tilde{f}_n\}$ in \eqref{problem 1} is supposed to be an orthonormal basis for $\mathcal{H}$  which diagonalizes $L$. 
However, diagonalizing $L$ can be extremely difficult in practical applications. We will consider redundant basis-like systems, and $L$ is not required to be diagonalized: a countable collection $\{f_n\}_{n\in\mathcal{N}}$ in $\mathcal{H}$ is a {\em frame} if there are two positive constants $A$, $B$ such that   
\begin{equation}\label{frame definition equation}
A\|g\|^2_\mathcal{H} \leq \sum_{n\in\mathcal{N}} |\langle g,f_n\rangle |^2\leq B\|g
\|^2_\mathcal{H}, \text{ for all $g\in \mathcal{H}$.}
\end{equation}
If $\{f_n\}_{n\in\mathcal{N}}$ is a frame, then its \emph{synthesis operator} 
\begin{equation}\label{eq:synthesis operator}
F:\ell_2(\mathcal{\mathcal{N}})\rightarrow \mathcal{H}, \quad (c_n)_{n\in \mathcal{N}}\mapsto \sum_{n\in \mathcal{N}}c_n f_n,
\end{equation}
is onto. Each $g\in \mathcal{H}$ then has a series expansion, but we still have to find its coefficients. The synthesis operator's adjoint 
\begin{equation}\label{def analysis operator in frame theory}
F^*:\mathcal{H}\rightarrow \l_2(\mathcal{N}), \quad g\mapsto (\langle g , f_n\rangle)_{n\in \mathcal{N}}
\end{equation}
is called \emph{analysis operator},
$S=FF^*$ 
is invertible, and $\{S^{-1}f_n\}_{n\in\mathcal{N}}$ is called {\em canonical dual frame} and expands 
\begin{equation*}
g=\sum_{n\in \mathcal{N}}\langle g, S^{-1}f_n\rangle f_n, \text{ for all $g\in \mathcal{H}$.}
\end{equation*}
The inversion of $S$ can be difficult, and, 
since $F$ need not be injective,
there could be `better' coefficients than $\langle g, S^{-1}f_n\rangle$. This motivates the following: two frames $\{f_n\}_{n\in\mathcal{N}}$ and $\{\tilde{f}_n\}_{n\in\mathcal{N}}$ are called a {\em pair of dual frames} (or a \emph{bi-frame}) if  
\begin{equation}\label{frame expansion in hilbert space}
g=\sum_{n\in \mathcal{N}}\langle g, \tilde{f}_n\rangle f_n, \text{ for all $g\in\mathcal{H}$,}
\end{equation}
i.e., $F\widetilde{F}^*=\id_\mathcal{H}$, where $\widetilde{F}^*$ is the dual frame's analysis operator. For instance, the canonical dual of a wavelet frame may not have the wavelet structure as well, but it can possibly replaced by an alternative dual wavelet frame, cf.~\cite{Daubechies:2000aa,ehler2,Ehler:2010ae,ehlerHan,Han:2009aa,Ron:1997aa} and references therein.

Throughout the paper, we suppose that $\{f_n\}_{n\in\mathcal{N}}$ and $\{\tilde{f}_n\}_{n\in\mathcal{N}}$ are a {\bf bi-frame for $\mathcal{H}$}.

\section{Shrinkage Rules}\label{shrinkage rules}
To solve \eqref{problem 1}, shrinkage plays a crucial role. Following ideas in \cite{tao1}, we call a function $\varrho:\C\times \R_{\geq 0}\rightarrow \C$ a {\em shrinkage rule} if there are constants $C_1, C_2, \rho, D>0$ such that both conditions
\begin{align}
 \left|x-\varrho(x,\alpha)\right|&\leq C_1 \min(|x|,\alpha), \text{ for all $\alpha\geq  0$, $x\in\C$,}\label{eq estimate threshold rule eins} \\
\left|\varrho(x,\alpha)\right|&\leq C_2 |x|\big|\frac{x}{\alpha}\big|^{\rho}, \text{ for all $\alpha> 0$, $|x|\leq D\alpha$,}\label{eq estimate threshold rule zwei}
\end{align}
are satisfied. While \eqref{eq estimate threshold rule eins} forces $\varrho(x,\alpha)$ to be close to $x$ for small $\alpha$, condition \eqref{eq estimate threshold rule zwei} means that $\varrho(x,\alpha)$ has sufficient decay as $x$ goes to zero.
A shrinkage rule $\varrho$ is called a \emph{thresholding rule} if there is a constant $C_3>0$ such that $|x|\leq C_3\alpha$ implies $\varrho(x,\alpha)=0$. A thresholding rule allows for $\rho=\infty$ in \eqref{eq estimate threshold rule zwei}, where we use $a^\infty=0$ if $0\leq a<1$. We will recall a few common shrinkage rules and we restrict us to $x\in\R$, see also Figure \ref{fig shrinkage rules}: \emph{Soft-shrinkage} is given by $\varrho_s(x,\alpha)=(x-\sign(x)\alpha){\bf 1}_{\{|x|>\alpha\}}$. Contrary to soft- and \emph{hard-shrinkage} 
$\varrho_h(x,\alpha)=x{\bf 1}_{\{\abs{x}>\alpha\}}$, the \emph{nonnegative garotte-shrinkage} rule $\varrho_g(x,\alpha)=(x-\frac{\alpha^2}{x}){\bf 1}_{\{\abs{x}>\alpha \}}$ is continuous and large coefficients are not changed much. It has been successfully applied to image denoising in \cite{gao}. Similar properties has \emph{hyperbolic-shrinkage} $\varrho^{hy}(x,\alpha)=\sign(x)\sqrt{x^2-\alpha^2} {\bf 1}_{\{\abs{x}>\alpha\}}(x)$, cf.~\cite{tao1}. 

The \emph{n-degree garotte-shrinkage} rule is given by $\varrho^n(x,\alpha)=\frac{x^{2n+1}}{x^{2n}+\alpha^{2n}}$, see \cite{tao1}. For $k\in\N$, the twice differentiable rule 
\begin{equation}\label{shrinkage 2k}
\varrho_k(x,\alpha)=\begin{cases}
\tfrac{x^{2k+1}}{(2k+1)\alpha^{2k}},& |x|\leq \alpha\\
x-\sign(x)(\alpha-\tfrac{\alpha}{2k+1}),& |x|> \alpha
\end{cases}
\end{equation}
is considered in \cite{surerisk}. Both rules are shrinkage rules with $\rho=2k=2n$ and $C_2=1$. The rules $\varrho(x,\alpha)=x\big(1-\sqrt{\frac{\alpha^2}{\alpha^2+2x^2}} \big)$ and $\varrho(x,\alpha)\approx x\exp(-0.2\frac{\alpha^8}{x^8})$ are based on diffusion, see \cite{WeickertSteidl}. One verifies that both are shrinkage rules with $\rho=1$, and we refer to them as diffusion 1 and 2 in Figure \ref{fig shrinkage rules}.

Bruce and Gao proposed \emph{firm-shrinkage} 
\begin{equation*}
\varrho_f(x,\alpha_1,\alpha_2)=x{\bf 1}_{\{|x|>\alpha_2\}}+\sign(x)\frac{\alpha_2(|x|-\alpha_1)}{\alpha_2-\alpha_1}{\bf 1}_{\{\alpha_1\leq |x|\leq \alpha_2\}}
\end{equation*}
in \cite{BruceGao}. For fixed $\alpha_1$, the mapping $(x,\alpha)\mapsto \varrho_f(x,\alpha_1,\alpha)$ is a thresholding rule. 

\newlength{\breit}
\setlength{\breit}{.2\textwidth}
\begin{figure}
\centering
\subfigure[hard]{
\includegraphics[width=\breit]{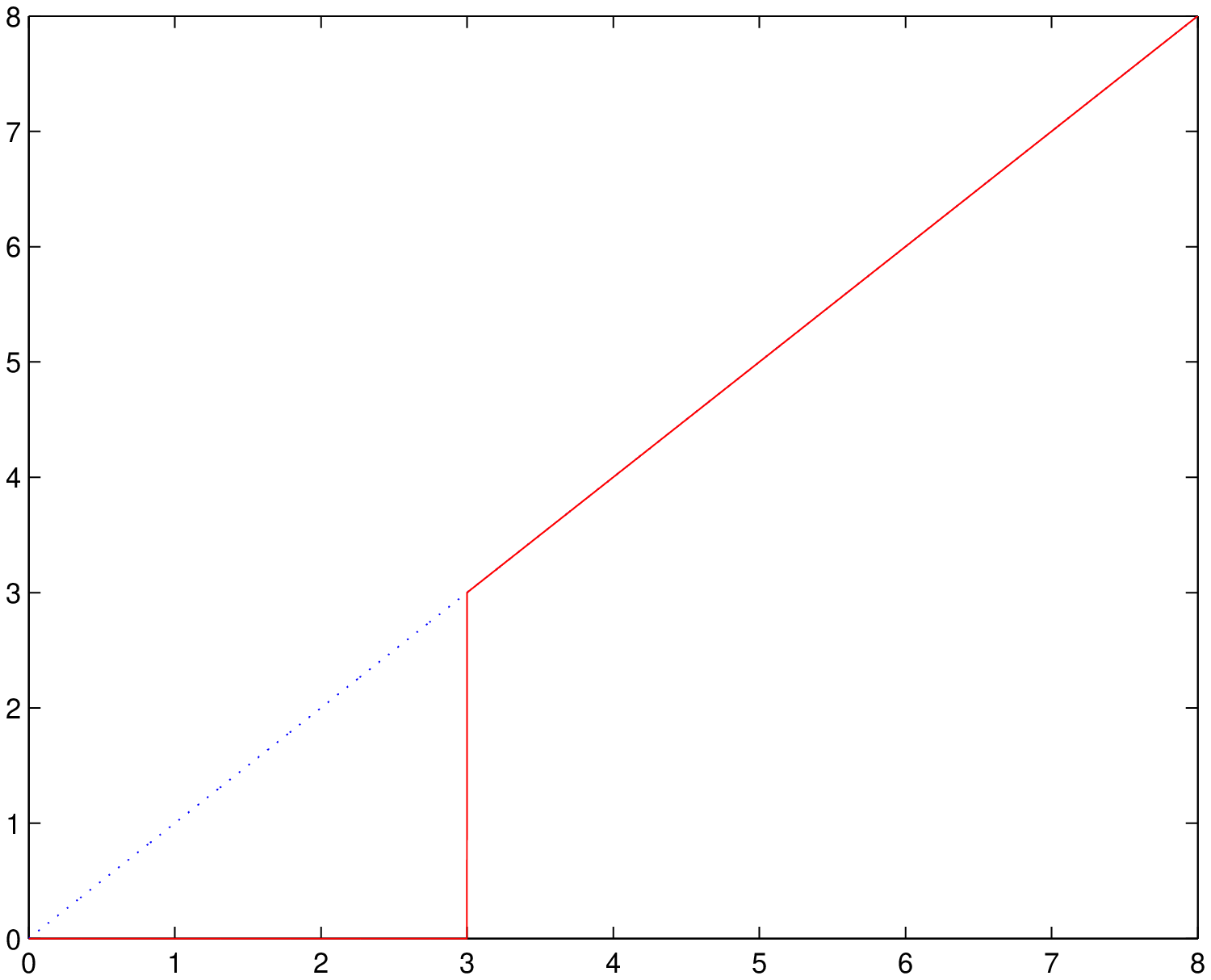}}
\subfigure[soft]{
\includegraphics[width=\breit]{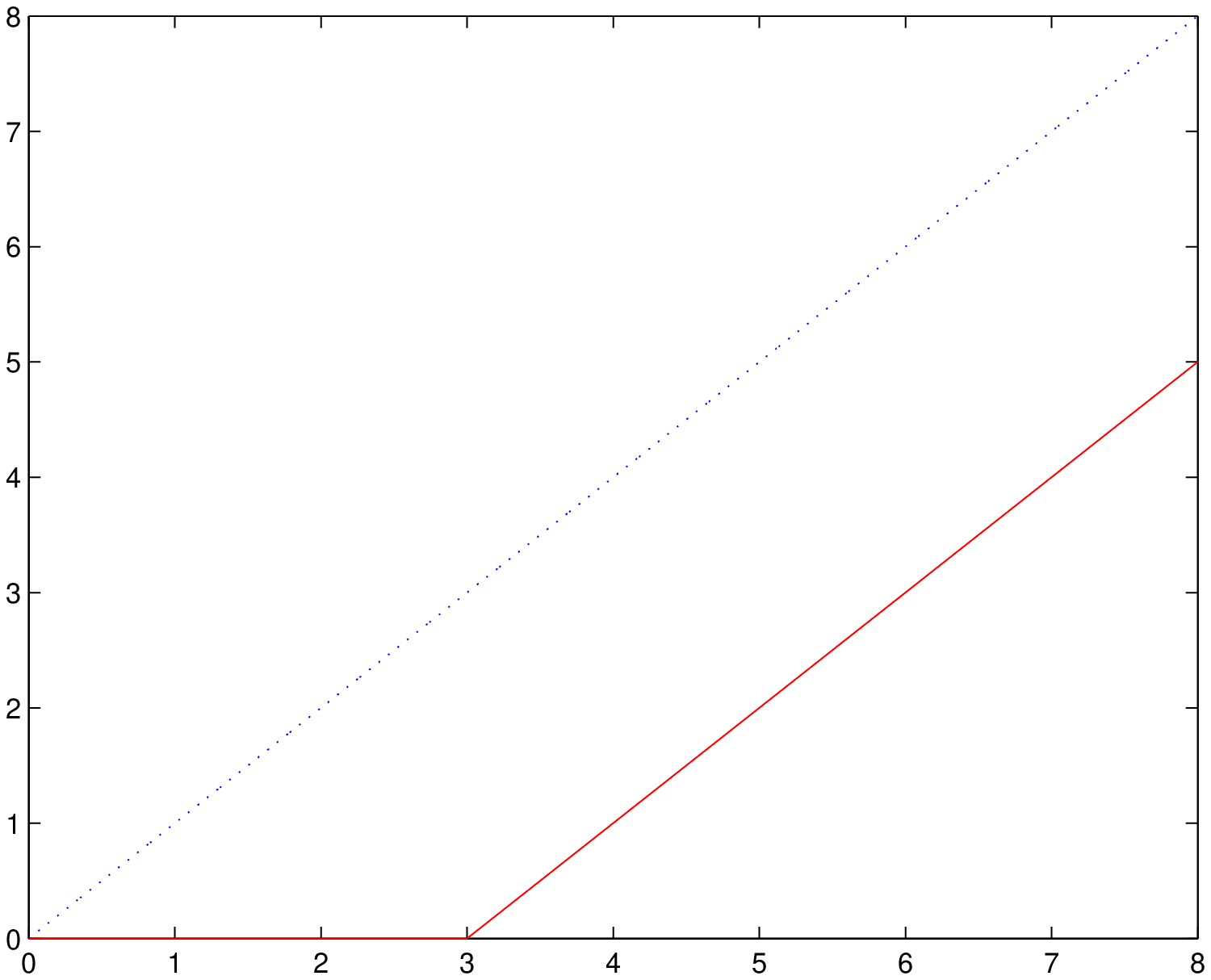}}
\subfigure[nonnegative garotte]{
\includegraphics[width=\breit]{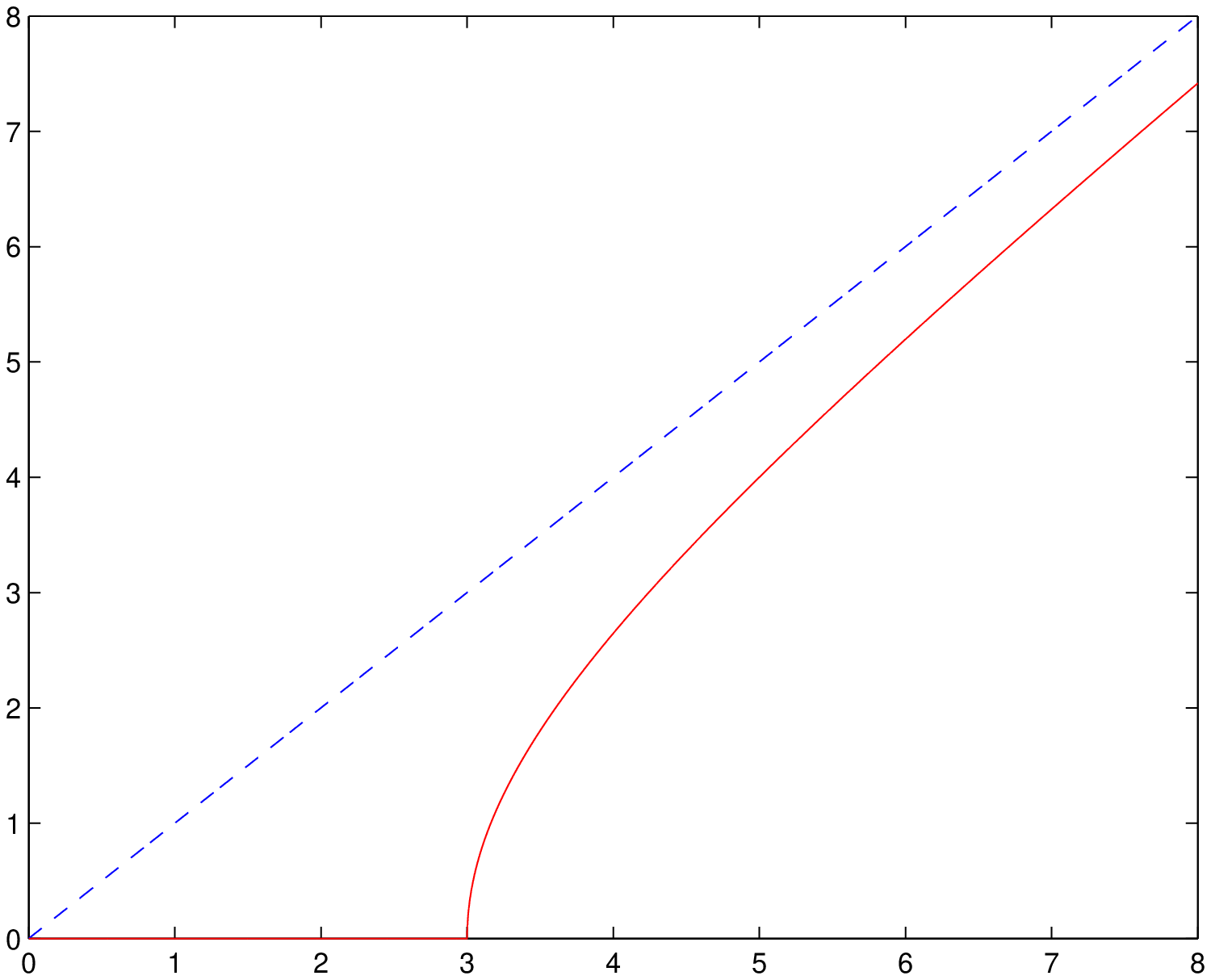}}
\subfigure[hyperbolic]{
\includegraphics[width=\breit]{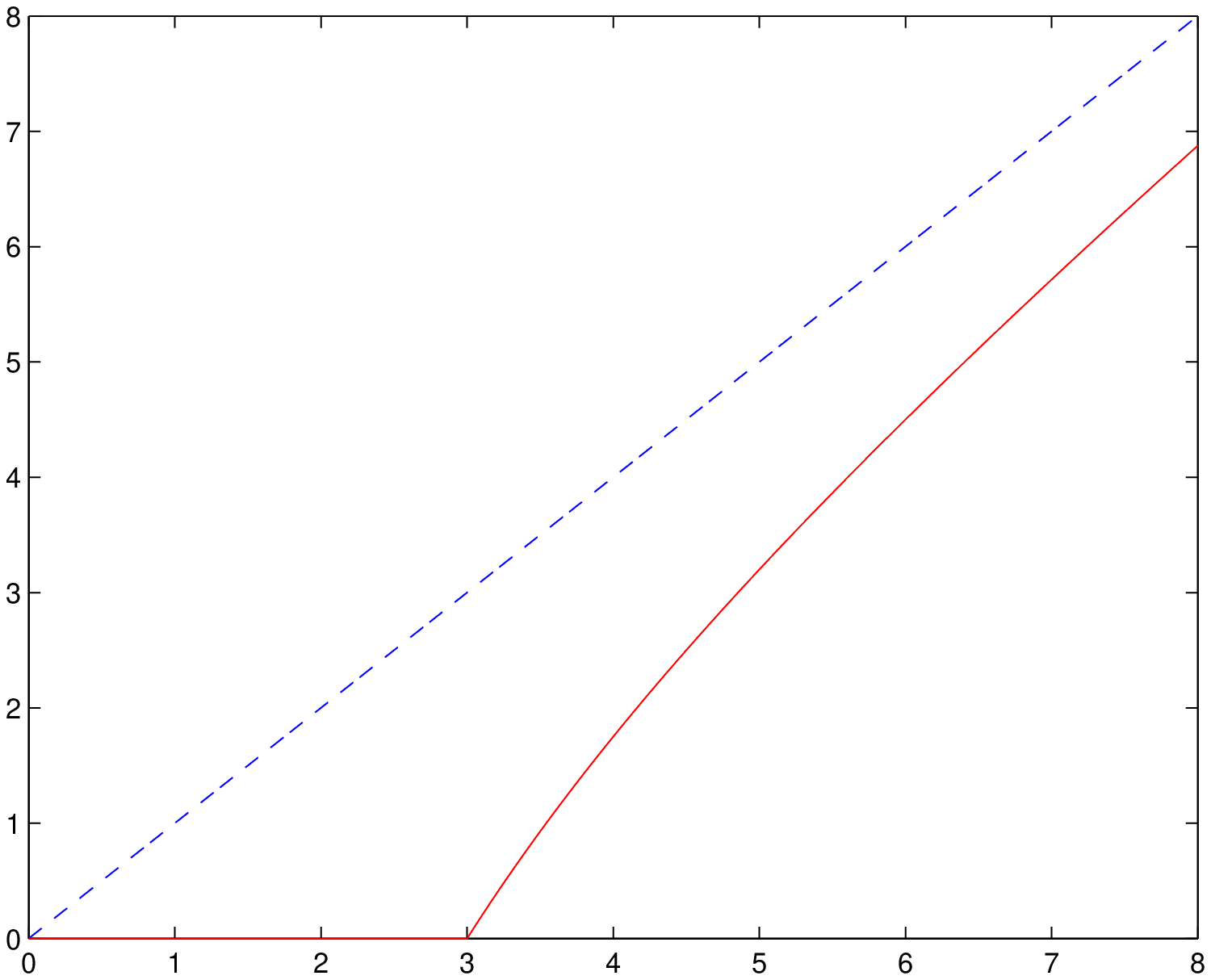}}

\subfigure[1-degree garotte]{
\includegraphics[width=\breit]{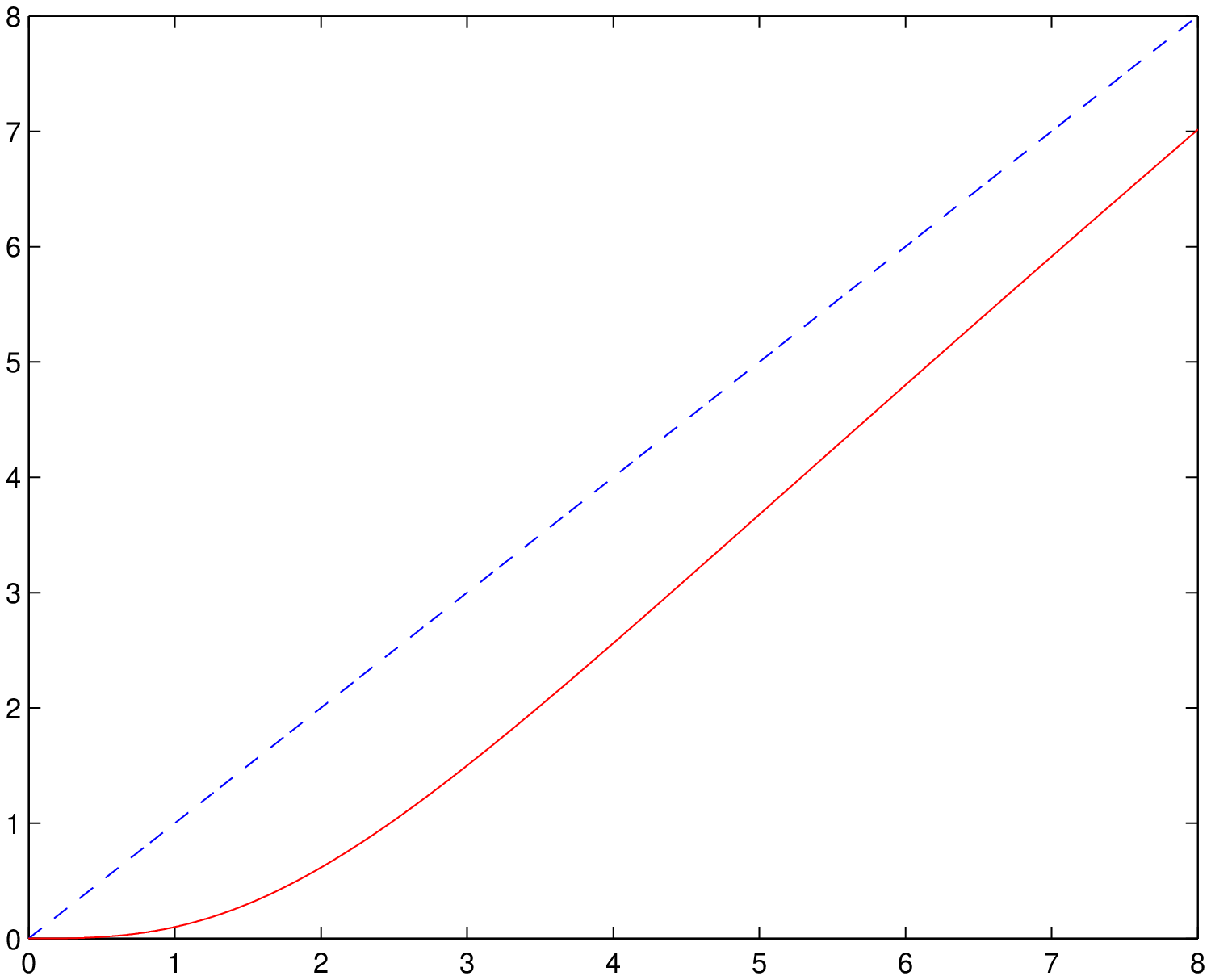}}
\subfigure[2-degree garotte]{
\includegraphics[width=\breit]{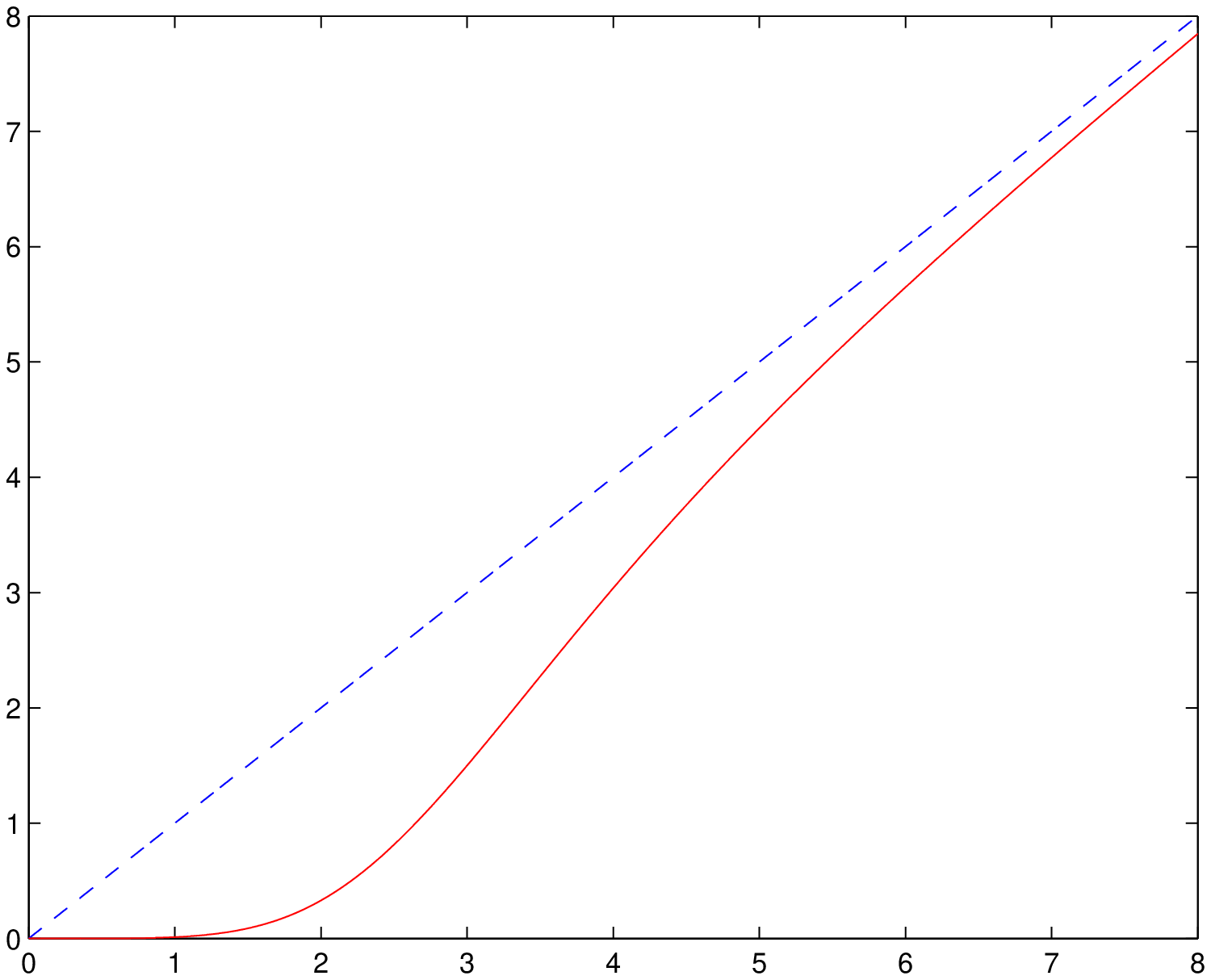}}
\subfigure[$k=1$-shrinkage]{
\includegraphics[width=\breit]{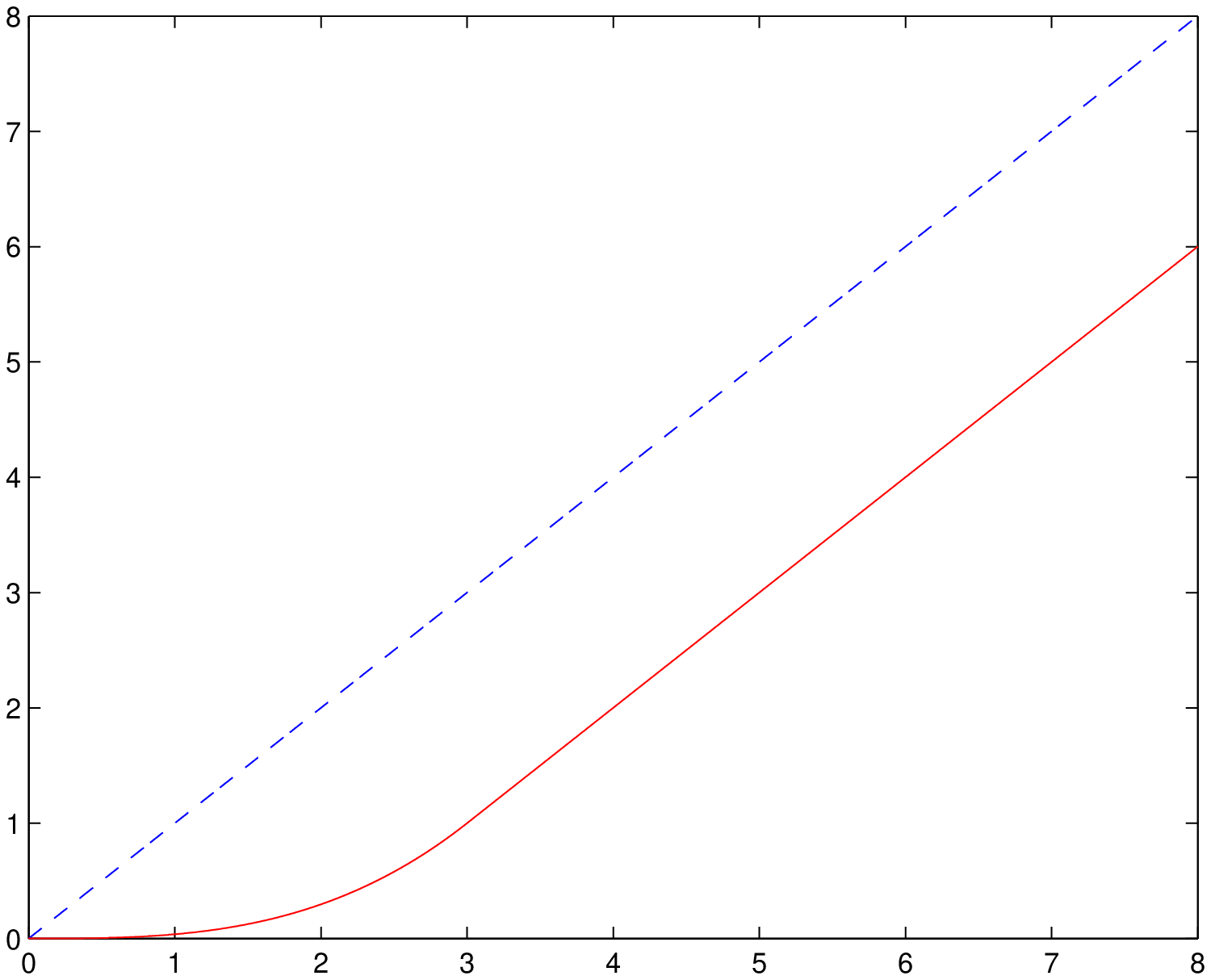}}
\subfigure[$k=2$-shrinkage]{
\includegraphics[width=\breit]{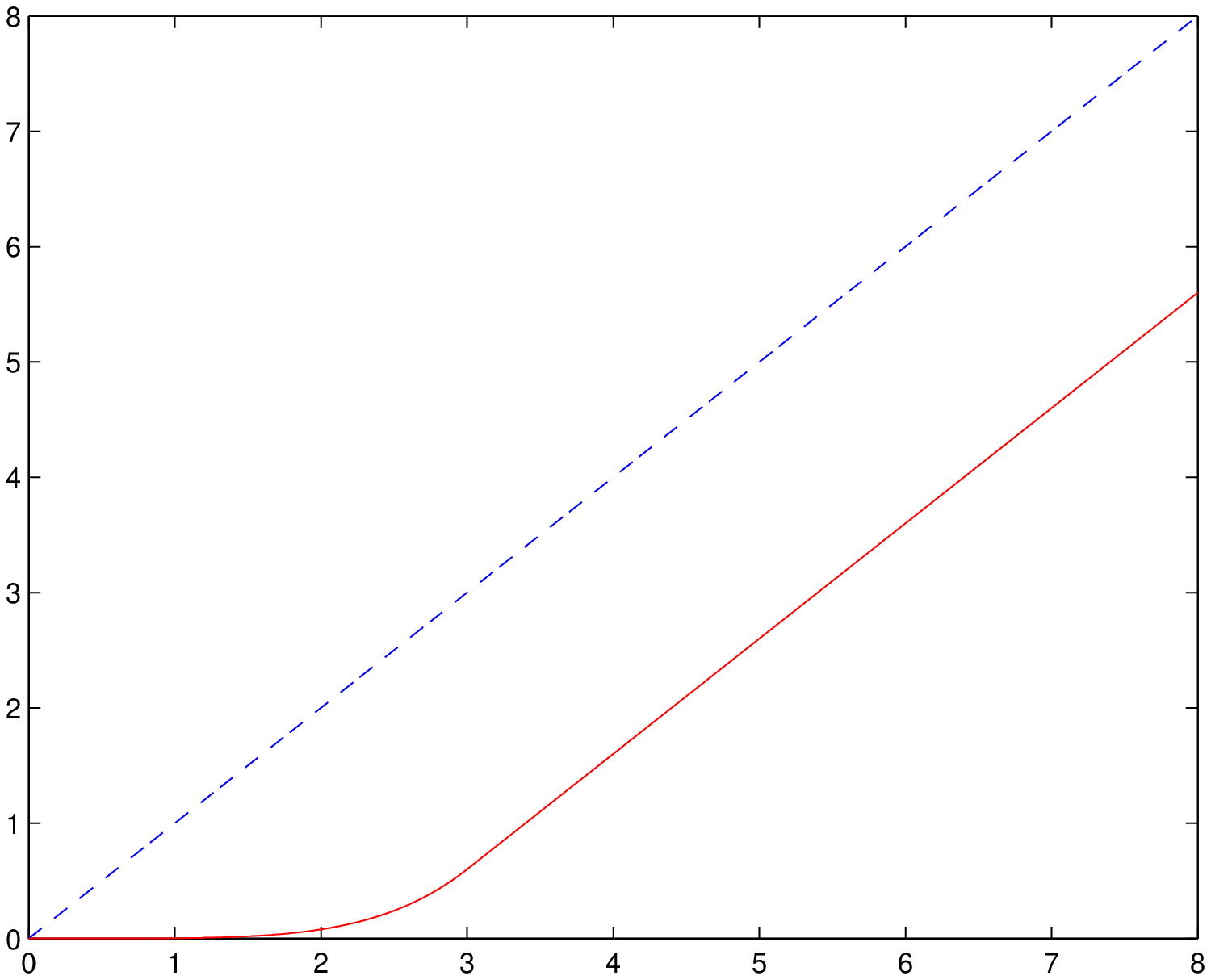}}

\subfigure[diffusion 1]{
\includegraphics[width=\breit]{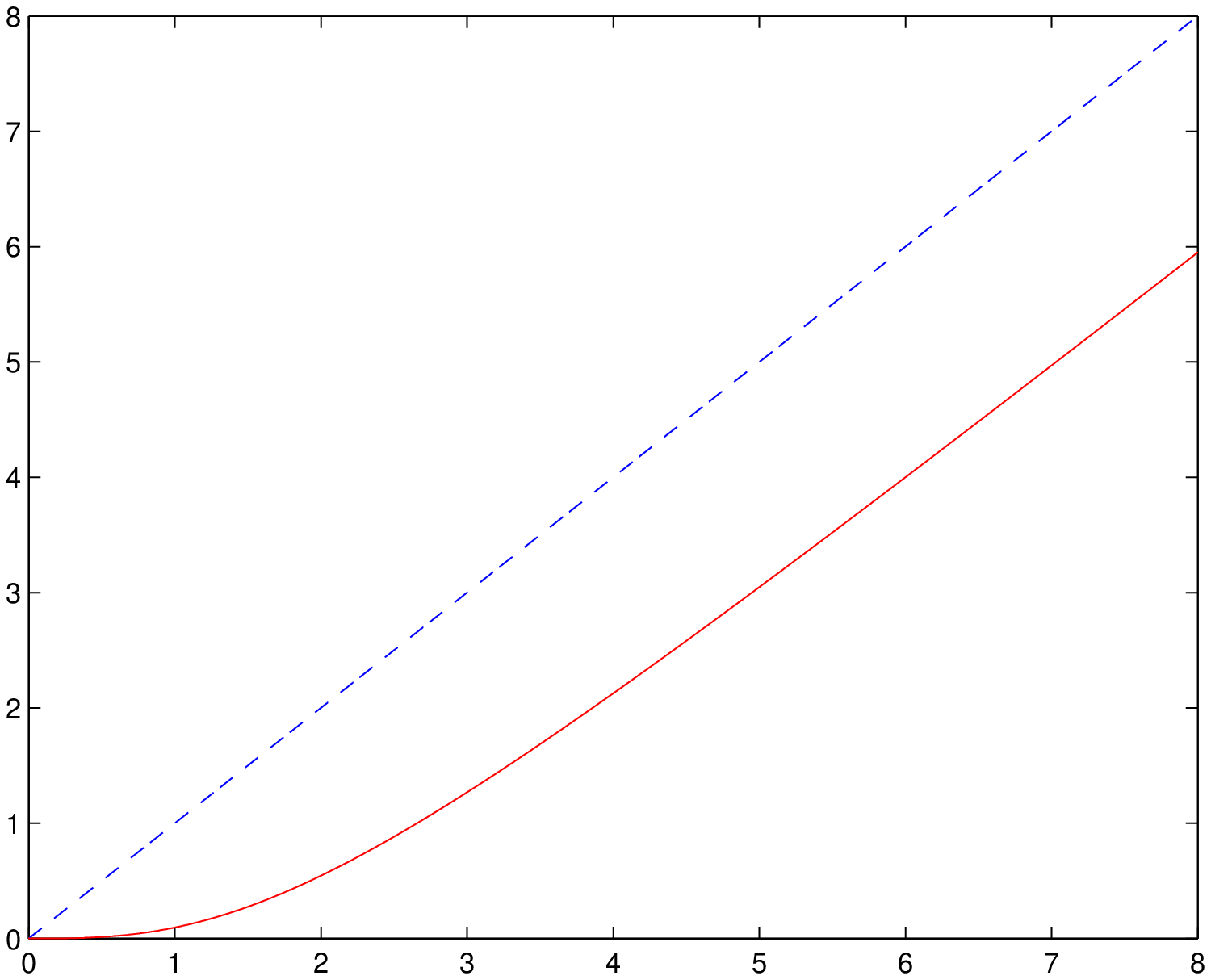}}
\subfigure[diffusion 2]{
\includegraphics[width=\breit]{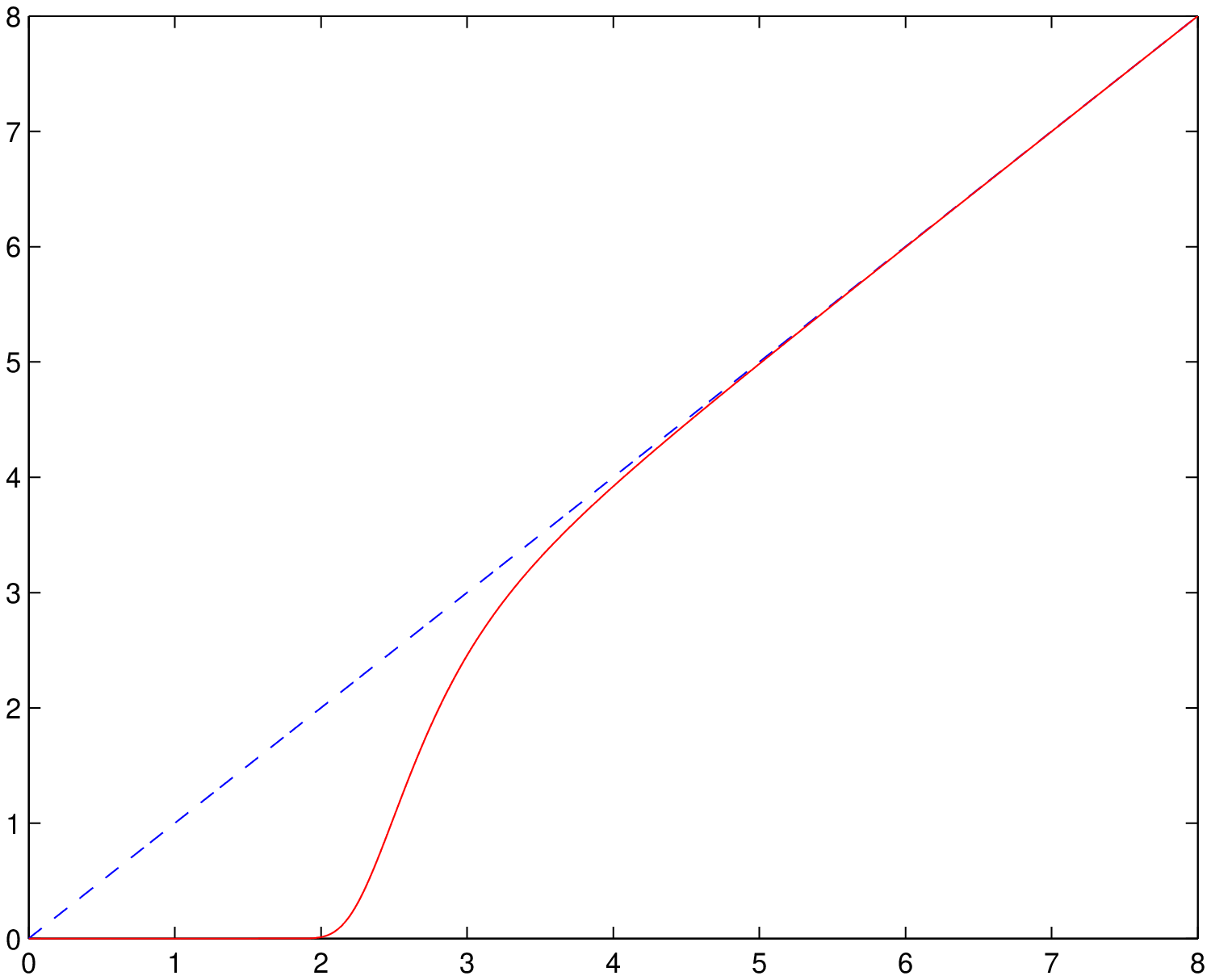}}
\subfigure[firm, $\alpha_1=1$]{
\includegraphics[width=\breit]{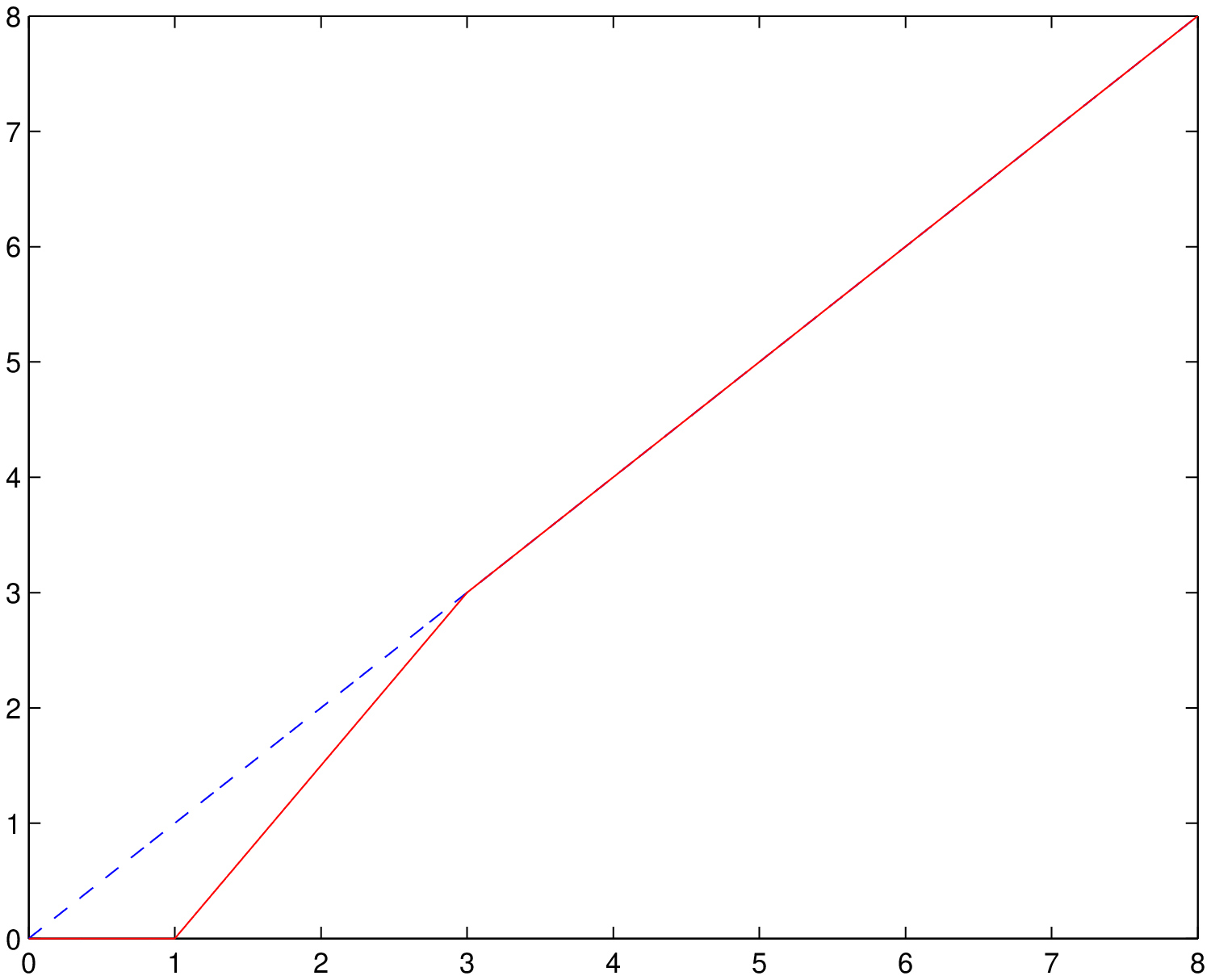}}
\subfigure[firm, $\alpha_1=2$]{
\includegraphics[width=\breit]{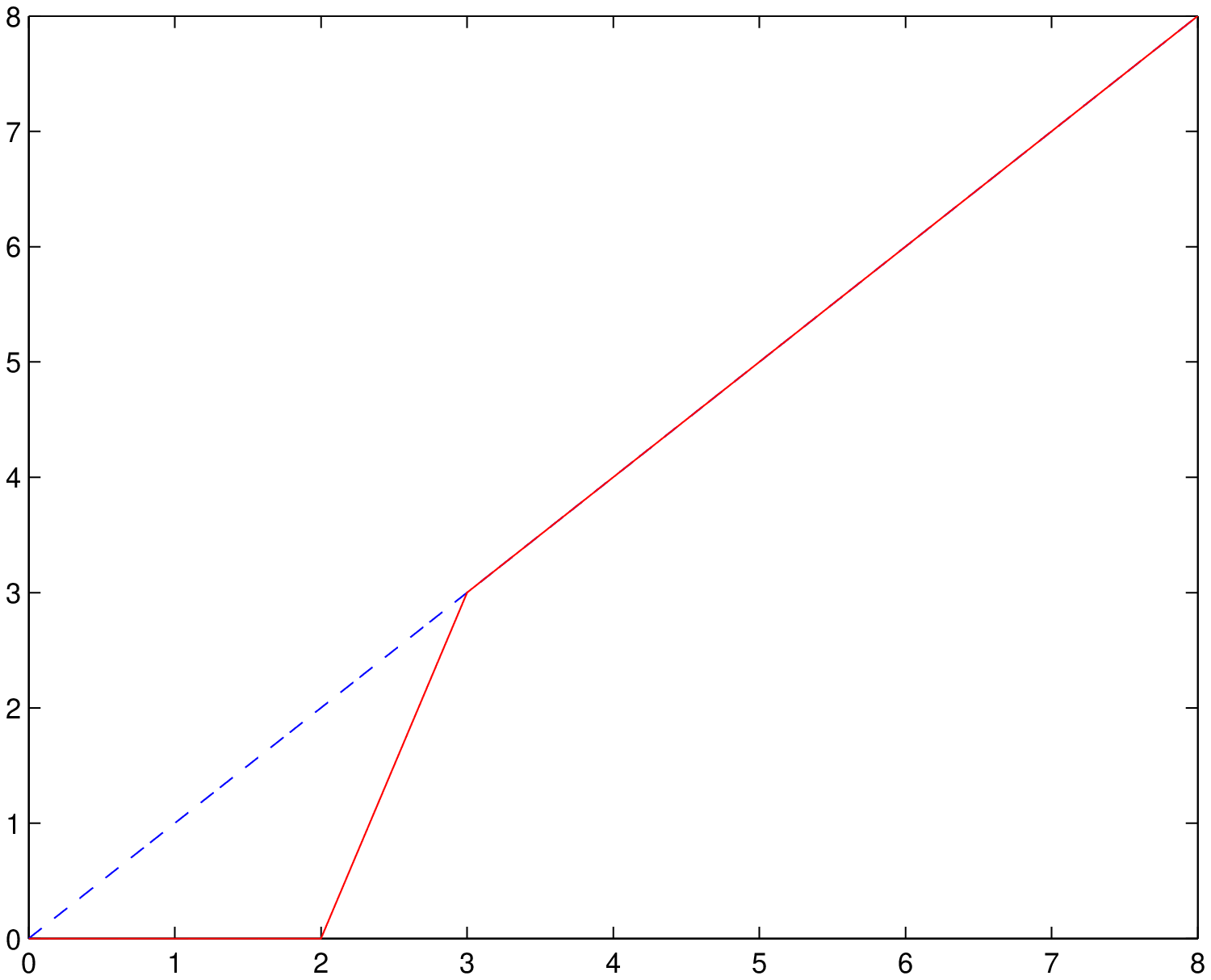}}
\caption{Shrinkage rules $\varrho(x,\alpha)$, for $\alpha=3$. (a) is not continuous. (b)-(d),(k), and (l) are continuous but not differentiable. (e)-(j) are smooth}\label{fig shrinkage rules}
\end{figure}

\section{Main Results}\label{section main results}
For $q\in[0,2]$, let $\ell^{(\alpha_n)}_q(\mathcal{N})$ denote the weighted $\ell_q(\mathcal{N})$-space, i.e., the space of complex-valued sequences $(\omega_n)_{n\in\mathcal{N}}$ such that 
$\|\omega\|^q_{\ell_q^{(\alpha_n)}}:=\sum_{n\in\mathcal{N}}\alpha_n |\omega_n|^q
$ is finite. One observes that $\sum_{n\in\mathcal{N}} \alpha_n |\langle g,\tilde{f}_n\rangle|^q=\|\widetilde{F}^*g\|^q_{\ell_q^{(\alpha_n)}}$, and to shorten notation, we denote 
\begin{equation*}
\mathcal{J}_q(h,g)= \|h-Lg\|^2_{\mathcal{H}'} + \|\widetilde{F}^* g\|_{\ell_q^{(\alpha_n)}}^q.
\end{equation*}
The idea for the following main result is to replace a shrinkage rule $\varrho(x,\alpha)$ by its $q$-dependent expression $\varrho(x,\alpha|x|^{q-1})$. Due to \eqref{eq estimate threshold rule eins}, it vanishes as $x\neq 0$ goes to $0$, and we apply $\varrho(x,\alpha|x|^{q-1})=0$ for $x=0$. If $\rho=\infty$, we use $\frac{1}{\rho}=0$. Since $\varrho_g(x,\alpha)=\varrho_s(x,\alpha^2|x|^{-1})$, the nonnegative garotte is $q$-dependent soft-shrinkage for $q=0$ and $\alpha$ replaced by $\alpha^2$. It turns out that $q$-dependent shrinkage expressions provide minimizers of \eqref{problem 1} up to a constant factor:
\begin{theorem}\label{theorem 1}
Let $\varrho$ be a shrinkage rule with $\rho\in[\frac{1}{2},\infty]$. Suppose that $\widetilde{F}^*L^\#LF$ is bounded on $\ell_{1/\rho}^{(\alpha_n)}(\mathcal{N})$. Let $q=\frac{1}{\rho}$, then there is a constant $C>0$ such that for all $h\in\range(L)$, and for all $g\in\mathcal{H}$
\begin{equation*}
\mathcal{J}_q(h,\hat{g})\leq C\mathcal{J}_q(h,g), 
\end{equation*}
where $\hat{g}=L^\#LF\varrho(v_n,\alpha_n|v_n|^{q-1})_{n\in\mathcal{N}}$ with $v=\widetilde{F}^*L^\#h$. 

If $\widetilde{F}^*F$ is also bounded on $\ell_{1/\rho}^{(\alpha_n)}(\mathcal{N})$, one can choose $\hat{g}=F\varrho(v_n,\alpha_n|v_n|^{q-1})_{n\in\mathcal{N}}$. If \eqref{eq:boundedness condition on delta} holds, then the statements extend to all $q\in[\frac{1}{\rho},2]$, and $C$ is independent of $q$.
\end{theorem}
\begin{remark}
If the bi-frame is biorthogonal and $F\ell_{1/\rho}^{(\alpha_n)}\subset \range(L^\#L)$, then $\widetilde{F}^*L^\#LF=\id_{\ell_{1/\rho}^{(\alpha_n)}}$, because $\widetilde{F}^*F$ is the identity and $L^\#L$ is the identity on its range. The boundedness condition is then trivially satisfied as it is for finite $\mathcal{N}$. 
\end{remark}
To prove Theorem \ref{theorem 1}, we consider a decoupled minimization problem: given $v\in\ell_2(\mathcal{N})$, we try to minimize 
\begin{equation}\label{problem equiv}
\mathcal{I}_q(v,\omega)=\big(\|v-\omega\|^2_{\ell_2} + \sum_{n\in\mathcal{N}} \alpha_n |\omega_n|^q\big)
\end{equation}
over $\omega\in\ell_2(\mathcal{N})$. It turns out that minimizing \eqref{problem 1} and \eqref{problem equiv} up to a constant factor are equivalent:
\begin{proposition}\label{proposition equiv 1}
Given $q\in[0,2]$, suppose that $\widetilde{F}^*L^\#LF$ is bounded on $\ell_q^{(\alpha_n)}(\mathcal{N})$. For $h\in\range(L)$, let $v=\widetilde{F}^*L^\#h$. If $\hat{\omega}$ minimizes \eqref{problem equiv} up to a constant factor, then $\hat{g}=L^\#LF\hat{\omega}$ minimizes \eqref{problem 1} up to a constant factor. 

If $\widetilde{F}^*F$ is bounded on $\ell_{q}^{(\alpha_n)}(\mathcal{N})$, one may also choose $\hat{g}=F\hat{\omega}$. The reverse implication holds for $\hat{\omega}=\widetilde{F}^*L^\#L\hat{g}$ and $\hat{\omega}=\widetilde{F}^*\hat{g}$, respectively. 
\end{proposition}
Given a parameter set $\Gamma$ and two expressions $(a_\tau)_{\tau\in \Gamma}$ and $(b_\tau)_{\tau\in \Gamma}$ such that there is a constant $C>0$ with $a_\tau\leq Cb_\tau$ for all $\tau\in \Gamma$, we write $a_\tau\lesssim b_\tau$ in the following proof.
\begin{proof}[Proof of Proposition \ref{proposition equiv 1}]
Let $\hat{\omega}$ minimize \eqref{problem equiv} up to a constant factor, i.e., $\mathcal{I}_q(v,\hat{\omega})\lesssim \mathcal{I}_q(v,\omega)$, for all $\omega\in\ell_2(\mathcal{N})$. Since $F\widetilde{F}^*=\id_\mathcal{H}$ and since $LL^\#L=L$ yields $LL^\#h=h$, we have $h=LFv$. Applying $LL^\#L=L$ implies $L\hat{g}=LF\hat{\omega}$, which leads to
\begin{equation*}
\mathcal{J}_q(h,\hat{g}) = \|LFv-LF\hat{\omega}\|^2_{\mathcal{H}'} + \| \widetilde{F}^*L^\#LF\hat{\omega} \|^q_{\ell_q^{(\alpha_n)}}.
\end{equation*}
Since $LF:\ell_2\mapsto \mathcal{H}'$ is bounded and due to the boundedness of $\widetilde{F}^*L^\#LF$ on $\ell_q^{(\alpha_n)}$, this implies $\mathcal{J}_q(h,\hat{g})\lesssim \mathcal{I}_q(v,\hat{\omega})$. Since $\hat{\omega}$ minimizes \eqref{problem equiv} up to a constant factor, we have $\mathcal{J}_q(h,\hat{g})\lesssim \mathcal{I}_q(v,\widetilde{F}^*L^\#Lg)$, for all $g\in\mathcal{H}$. By applying that $\widetilde{F}^*L^\#$ is bounded 
and that $F\widetilde{F}^*=\id_\mathcal{H}$, we obtain, for all $g\in\mathcal{H}$,
\begin{align*}
\mathcal{J}_q(h,\hat{g}) &\lesssim \|\widetilde{F}^*L^\#h-\widetilde{F}^*L^\#Lg\|^2_{\ell_2} + \| \widetilde{F}^*L^\#LF\widetilde{F}^*g \|^q_{\ell_q^{(\alpha_n)}} \\
& \lesssim \| h-Lg\|^2_{\ell_2} + \| \widetilde{F}^*L^\#LF\widetilde{F}^*g \|^q_{\ell_q^{(\alpha_n)}} \lesssim \mathcal{J}_q(h,g),
\end{align*}
where we have used that $\widetilde{F}^*L^\#LF$ is bounded on $\ell_q^{(\alpha_n)}$. 

Analogous arguments can be applied to the case $\hat{g}=F\hat{\omega}$, and the reverse implications follow in a similar way.
\end{proof}
Next, we obtain a solution of the discrete problem \eqref{problem equiv}.
\begin{proposition}\label{proposition solving equiv}
Let $\varrho$ be a shrinkage rule with $\rho\in[\frac{1}{2},\infty]$. Then there is a constant $C>0$ such that for all $q\in[\frac{1}{\rho},2]$, for all $v\in\ell_2(\mathcal{N})$, and for all $\omega\in\ell_2(\mathcal{N})$,
\begin{equation*}
\mathcal{I}_q(v,\hat{\omega})\leq C \mathcal{I}_q(v,\omega),
\end{equation*}
where $\hat{\omega}=\varrho(v_n,\alpha_n|v_n|^{q-1})_{n\in\mathcal{N}}$.
\end{proposition}
\begin{remark}
The exact minimizer of \eqref{problem equiv} for $q=2$ is known to be $\big(\frac{1}{1+\alpha_n}v_n\big)_{n\in\mathcal{N}}$. However, $(x,\alpha)\mapsto \frac{1}{1+\alpha}x$ is not a shrinkage rule since \eqref{eq estimate threshold rule zwei} is violated. On the other hand, the rule $\varrho(x,\alpha)=\frac{1}{1+\frac{\alpha}{|x|}}x$ is a shrinkage rule with constant $\rho=1$. The $q$-dependent expression $\varrho(x,\alpha|x|^{q-1})$ for $q=2$ then yields the exact minimizer. In this sense the exact minimizer for $q=2$ is still derived from shrinkage. 
\end{remark}
\begin{proof}[Proof of Proposition \ref{proposition solving equiv}]
First, we consider $\frac{1}{2}\leq \rho<\infty$. Due to \eqref{eq estimate threshold rule eins}, the sequence $\varrho(v_n,\alpha_n|v_n|^{q-1})_{n\in\mathcal{N}}$ is indeed contained in $\ell_2(\mathcal{N})$. 
Adapting results in \cite{variational_denoising} to our setting yields that the hard-shrinked sequence $\varrho_h(v_n,\alpha_n|v_n|^{q-1} )_{n\in\mathcal{N}}$ minimizes \eqref{problem equiv} up to a constant factor. By using the short-hand notation
\begin{align*}
K_n & := \big|v_n-\varrho_h(v_n,\alpha_n|v_n|^{q-1})\big|^2+\alpha_n \big|\varrho_h(v_n,\alpha_n|v_n|^{q-1})\big|^q,\\
G_n &:= |v_n-\varrho(v_n,\alpha_n|v_n|^{q-1})|^2+\alpha_n |\varrho(v_n,\alpha_n|v_n|^{q-1})|^q,
\end{align*}
we consider each $n$ in the sequence norms separately. We aim to verify $G_n\lesssim K_n$ independently of $n$. For $v_n=0$, we have $G_n=K_n$. Now, we suppose $v_n\neq 0$. Since \eqref{eq estimate threshold rule zwei} gets weaker as $\rho$ and $D$ decrease, we may assume that $q=\frac{1}{\rho}$ and $D\leq 1$. Case 1: For $|v_n|\leq D\alpha_n|v_n|^{q-1}$, \eqref{eq estimate threshold rule eins} and \eqref{eq estimate threshold rule zwei} with $\rho=\frac{1}{q}$ yield
\begin{align*}
G_n & \leq C_1^2 |v_n|^2+\alpha_n C_2^q|v_n|^q \frac{|v_n|}{\alpha_n |v_n|^{q-1}}\\
&\leq C_1^2 |v_n|^2+C_2^q|v_n|^2\lesssim |v_n|^2= K_n.
\end{align*}
Case 2: For $|v_n|> D\alpha_n|v_n|^{q-1}$, we have $1/D> \alpha_n|v_n |^{q-2}$, and the estimate \eqref{eq estimate threshold rule eins} yields
\begin{align*}
G_n&\leq C_1^2(\alpha_n|v_n|^{q-1})^2 +\alpha_n \big(|v_n|+C_1\min(|v_n|,\alpha_n |v_n|^{q-1})\big)^q\\
&\leq C_1^2 \alpha_n|v_n|^q\alpha_n |v_n|^{q-2}+\alpha_n|v_n|^q(1+C_1\alpha_n |v_n|^{q-2})^q\\
&\leq C_1^2 \alpha|v_n|^q \tfrac{1}{D}+(1+C_1/D)^q\alpha_n|v_n|^q\lesssim \alpha_n|v_n|^q \leq K_n/D.
\end{align*}
Hence, $G_n\lesssim K_n$ holds in both cases. Similar arguments verify the statement for $\rho=\infty$.
\end{proof}
Our main result follows from combining both propositions:
\begin{proof}[Proof of Theorem \ref{theorem 1}]
According to Proposition \ref{proposition solving equiv}, $\varrho(v_n,\alpha_n|v_n|^{q-1})_{n\in\mathcal{N}}$ is a minimizer of \eqref{problem equiv} up to a constant factor, where $v=\widetilde{F}^*L^\#f$. For $q=\frac{1}{\rho}$, Proposition \ref{proposition equiv 1} then implies Theorem \ref{theorem 1}. If \eqref{eq:boundedness condition on delta} holds, $\widetilde{F}^*L^\#LF$ and $\widetilde{F}^*F$ are bounded on $\ell_2^{(\alpha_n)}$. Interpolation between $\ell_{1/\rho}^{(\alpha_n)}$ and $\ell_2^{(\alpha_n)}$ yields uniform boundedness on $\ell_q^{(\alpha_n)}$, for $q\in[\frac{1}{\rho},2]$.
\end{proof}
\begin{remark}
We did not use the Hilbert space structure of $\mathcal{H}'$ and in fact Theorem \ref{theorem 1} still holds if $\mathcal{H}'$ is a (quasi) Banach space.
\end{remark}

\section{Sparse Approximation}\label{section:sparse approximation}
Given $h\in\mathcal{H}$ (possibly noisy) and a frame $\{f_n\}_{n\in\mathcal{N}}$ for $\mathcal{H}$, an important problem in sparse signal representation is to find the minimizer of 
\begin{equation}\label{eq:classic sparse}
\min_{\omega\in\ell_2}\|\omega\|_{\ell_q} \text{ subject to } F\omega\approx h,
\end{equation}
for $q\in[0,1)$. Under additional requirements on $\{f_n\}_{n\in\mathcal{N}}$ and $h$, the solution for $q\in[0,1)$ can be obtained from solving the much simpler convex problem with $q=1$, cf.~\cite{CanRomTao,DonElaTem}. However, these results are limited to finite $\mathcal{N}$, and the additional requirements are not satisfied in many situations.
 
The problem \eqref{eq:classic sparse} is often replaced by a variational formulation, and one seeks to minimize
\begin{equation*}
\mathcal{K}_q(h,\omega)=\|h-F\omega\|^2_{\mathcal{H}}+\sum_{n\in\mathcal{N}}\alpha_n|\omega_n|^q\end{equation*}
over $\omega\in\ell_2(\mathcal{N})$. For finite $\mathcal{N}$, $\ell_q$-basis-pursuit as in \cite{Yilmaz}, for instance, solves \eqref{eq:classic sparse} by minimizing $\|F^\#h\|_{\ell_q}$ over all pseudo inverses $F^\#$. The associated variational formulation is
\begin{equation*}
\min_{F^\#}\big(\min_{\omega\in\ell_2}\big(\|F^\#h-\omega\|_{\ell_2} + \sum_{n\in\mathcal{N}}\alpha_n|\omega_n|^q\big)\big).
\end{equation*}
We do not require $\mathcal{N}$ to be finite, and instead of minimizing over $F^\#$, we suppose to have a particular pseudo inverse $\widetilde{F}^*$ being the analysis operator of a dual frame $\{\tilde{f}_n\}_{n\in\mathcal{N}}$ such that $\widetilde{F}^*F$ is bounded on $\ell_q^{(\alpha_n)}(\mathcal{N})$: 
\begin{theorem}\label{theorem 2}
Given a bi-frame $\{f_n\}_{n\in\mathcal{N}}$ and $\{\tilde{f}_n\}_{n\in\mathcal{N}}$, let $\varrho$ be a shrinkage rule with $\rho\in[\frac{1}{2},\infty]$. Suppose that $\widetilde{F}^*F$ is bounded on $\ell_{1/\rho}^{(\alpha_n)}(\mathcal{N})$. Let $q=\frac{1}{\rho}$, then there is a constant $C>0$ such that for all $h\in\mathcal{H}$ and for all $\omega\in\ell_2(\mathcal{N})$ 
\begin{equation*}
\mathcal{K}_q(h,\hat{\omega})\leq C\mathcal{K}_q(h,\omega),
\end{equation*}
where $\hat{\omega}=\widetilde{F}^*F\varrho(v_n,\alpha_n|v_n|^{q-1})_{n\in\mathcal{N}}$ with $v=\widetilde{F}^*h$ or $\hat{\omega}=\varrho(v_n,\alpha_n|v_n|^{q-1})_{n\in\mathcal{N}}$. If \eqref{eq:boundedness condition on delta} holds, then the statement extends to all $q\in[\frac{1}{\rho},2]$, and $C$ is independent of $q$.
\end{theorem}
\begin{remark}
For sufficiently smooth wavelet bi-frames with vanishing moments, the operator $\widetilde{F}^*F$ is bounded on $\ell_{1/\rho}^{(\alpha_n)}$ provided that $(\alpha_n)_{n\in\mathcal{N}}$ satisfies \eqref{eq:boundedness condition on delta}, cf.~\cite{ehler3}.
\end{remark}
\begin{proof}
We replace $\mathcal{H}$, $\mathcal{H}'$, $L$, $L^\#$, and the bi-frame $\{f_n\}_{n\in\mathcal{N}}$, $\{\tilde{f}_n\}_{n\in\mathcal{N}}$ in \eqref{problem 1} by $\ell_2(\mathcal{N})$, $\mathcal{H}$, $F$, $\widetilde{F}^*$, and the canonical basis $\{e_n\}_{n\in\mathcal{N}}$ for $\ell_2(\mathcal{N})$, respectively. The condition on $\widetilde{F}^*L^\#LF$ in Theorem \ref{theorem 1} becomes `$\widetilde{F}^*F$ is bounded on $\ell_{1/\rho}^{(\alpha_n)}(\mathcal{N})$', and Theorem \ref{theorem 1} implies Theorem \ref{theorem 2}. 
\end{proof}

\section{Explicit Shrinkage Rules Between Hard- and Soft-Shrinkage}\label{section interpolation} 
This section is dedicated to finding a family of shrinkage rules which is adapted to $q$ in \eqref{problem equiv}. For $q\in [0,1)$, we will use $c_q=2^{q-2}\frac{(2-q)^{2-q}}{(1-q)^{1-q}}$. It is monotonically decreasing with $c_0=1$, and continuous extension yields $c_1=\frac{1}{2}$, see Figure \ref{fig cq}. 

\begin{figure}
\centering
\includegraphics[width=.25\textwidth]{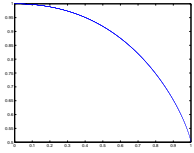}
\caption{The curve $c_q$, for $q\in [0,1]$, is continuous and continuation yields $c_0=1$ and $c_1=\frac{1}{2}$.}\label{fig cq}
\end{figure}
Due to \cite{Antoniadis01regularizationof}, the exact minimizer of \eqref{problem equiv} is sandwiched between soft- and hard-shrinkage. Let us introduce the new shrinkage rule
\begin{equation}\label{eq:new shrinkage rule}
\varrho^{(q)}_{h,s}(x,\alpha)=(x-\sign(x) q c_q \alpha){\bf 1}_{\{|x|>\alpha c_q\}},
\end{equation}
see Figure \ref{fig new shrinkage rule}.

\begin{figure}
\centering
\subfigure[$q=0.1$]{
\includegraphics[width=.237\textwidth]{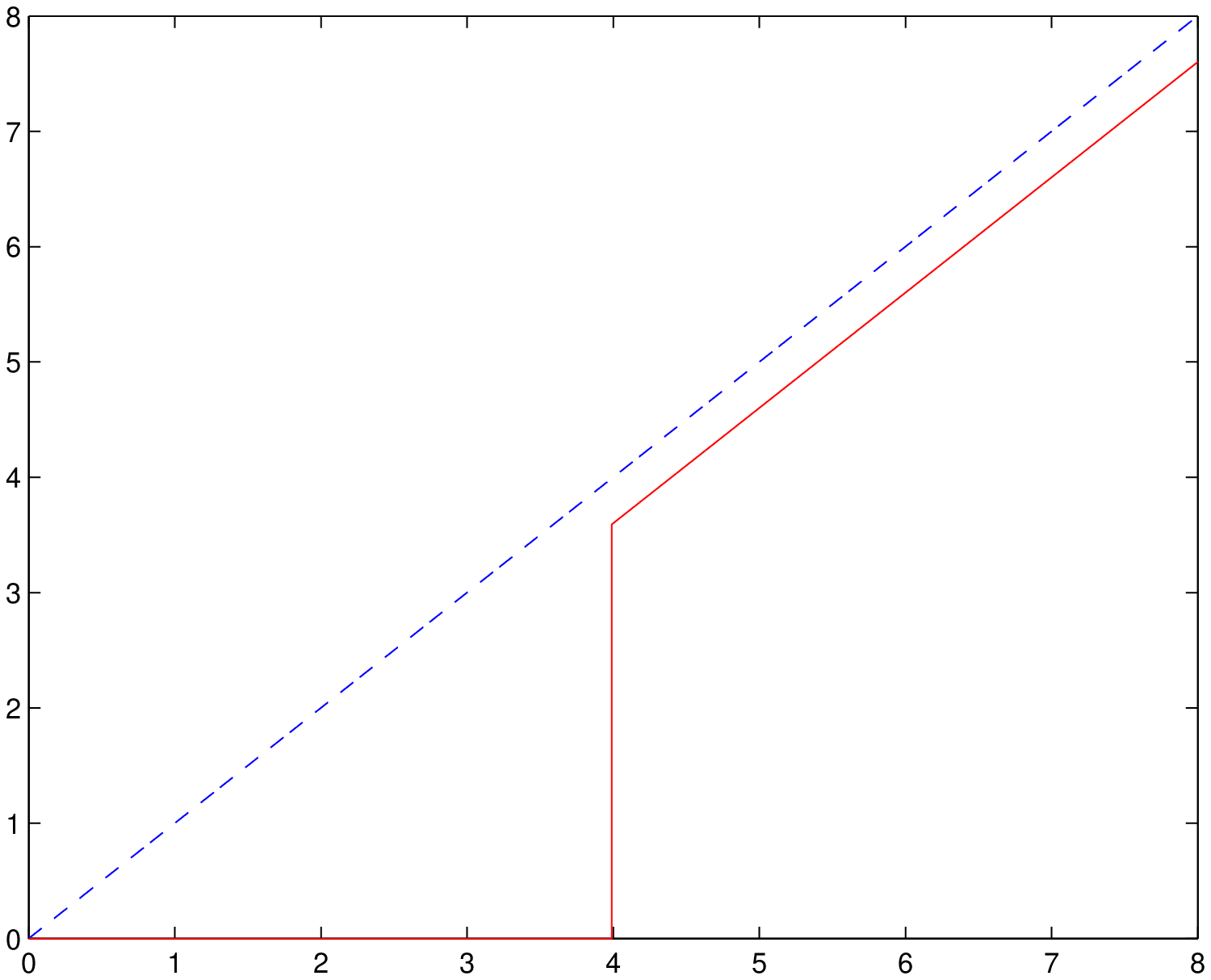}}
\subfigure[$q=0.3$]{
\includegraphics[width=.237\textwidth]{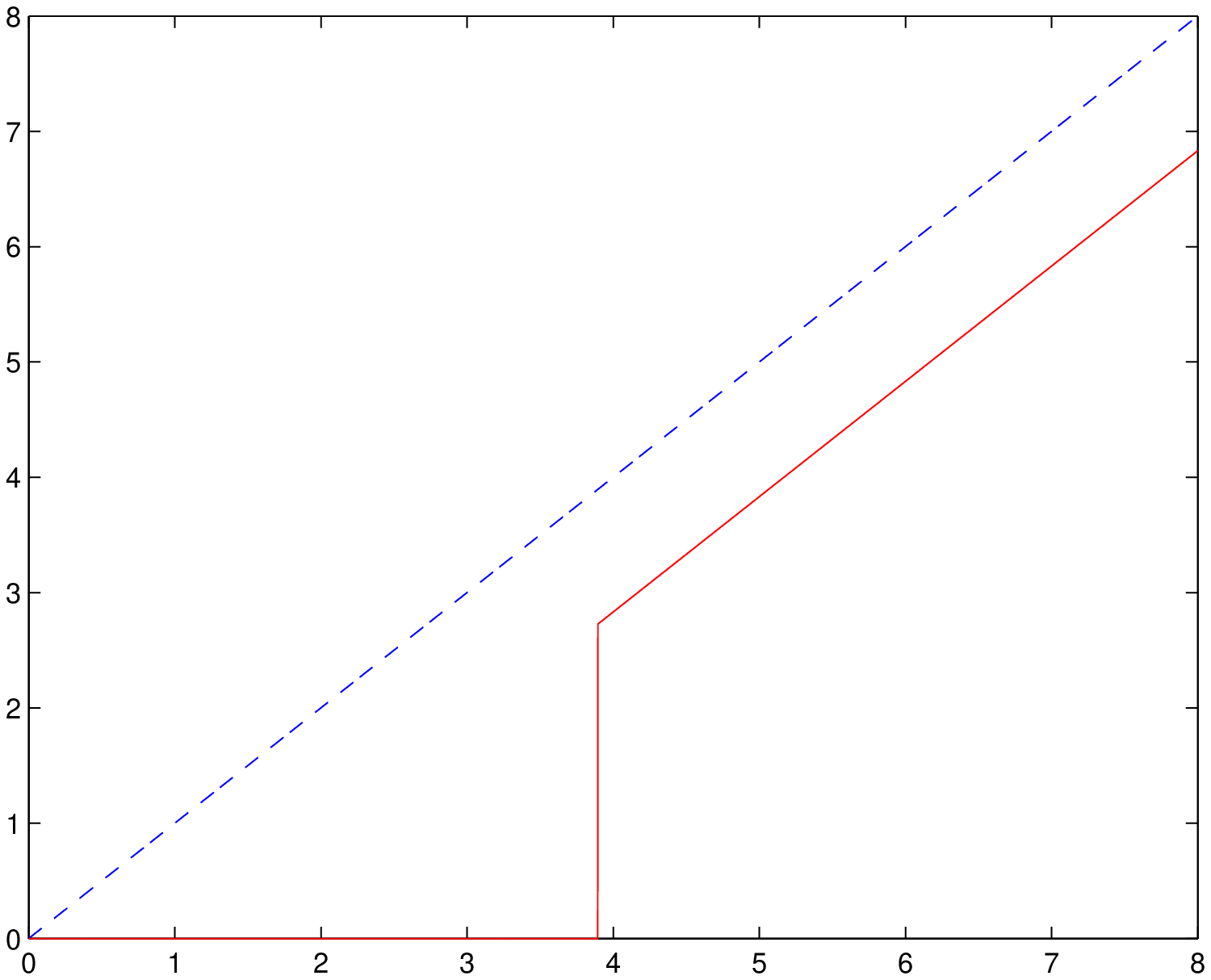}}
\subfigure[$q=1/2$]{
\includegraphics[width=.237\textwidth]{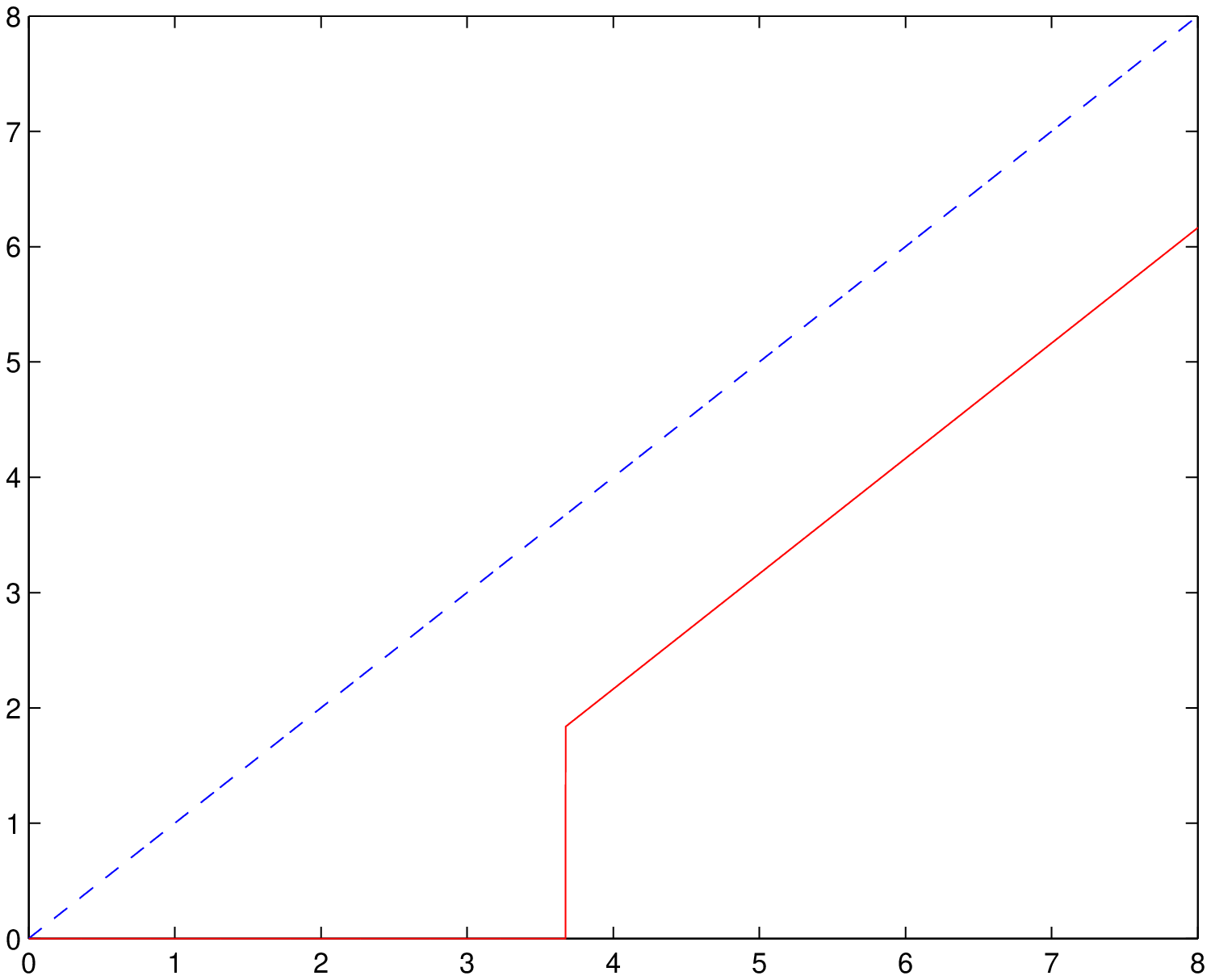}}
\subfigure[$q=0.8$]{
\includegraphics[width=.237\textwidth]{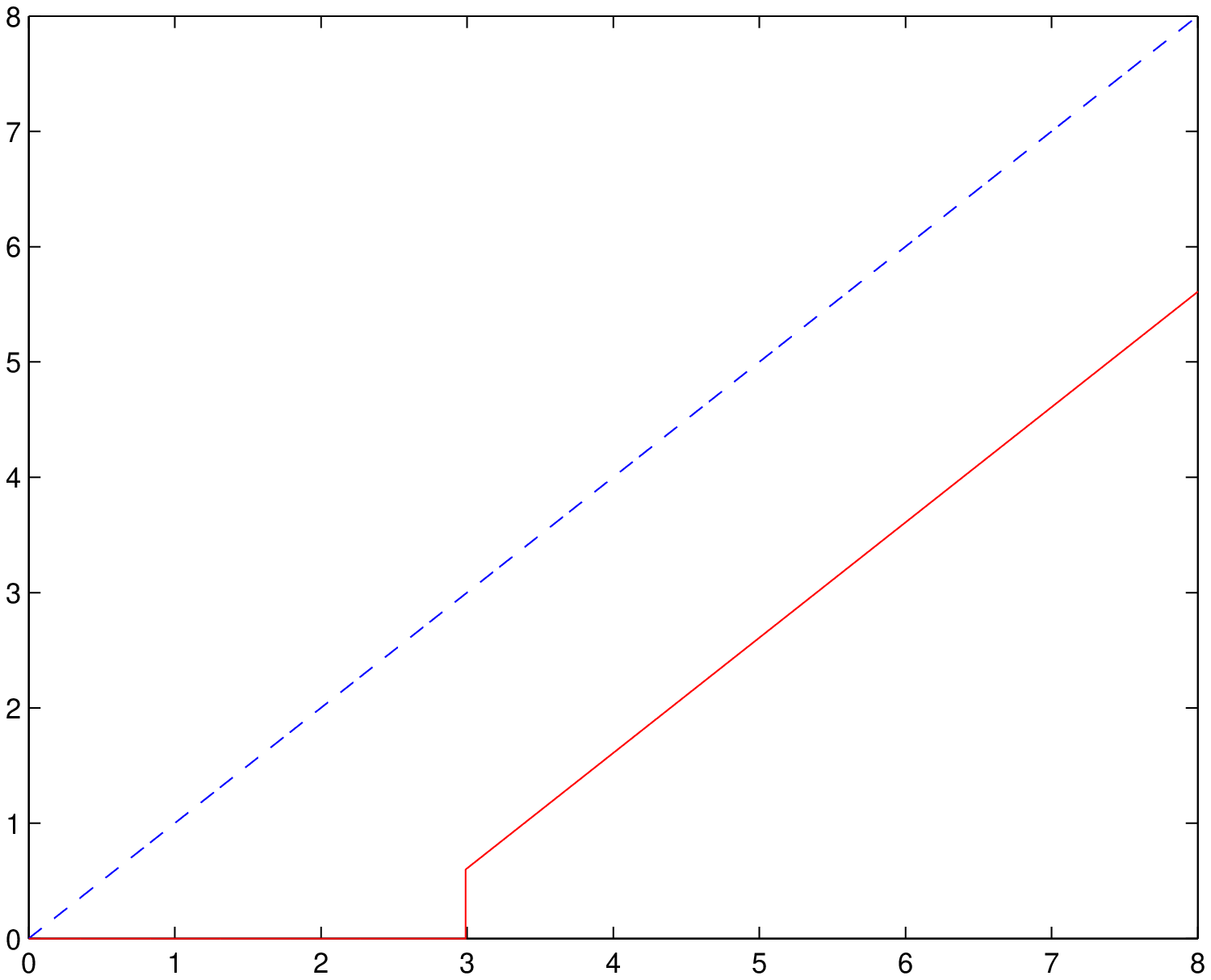}}
\caption{Shrinkage rule $\varrho^{(q)}_{h,s}(x,\alpha)$ for $\alpha=4$. It tends to hard-shrinkage for $q\searrow 0$. Soft-shrinkage is approximated by $q\nearrow 1$.}\label{fig new shrinkage rule}
\end{figure}
One easily verifies that $\varrho^{(q)}_{h,s}(x,\alpha)$ has a jump of size $(1-q)c_q\alpha$ and, for $|x|>c_q\alpha$, we have 
\begin{equation*}
|x-\varrho^{(q)}_{h,s}(x,\alpha)|=qc_q\alpha.
\end{equation*} 
On the other hand, the $q$-dependent expression $\varrho^{(q)}_{h,s}(x,\alpha|x|^{q-1})$, see Figure \ref{fig new shrinkage rule q dependent}, has a jump of size $(1-q)(c_q\alpha)^{\frac{1}{2-q}}$ and, for $|x|>(c_q\alpha)^{\frac{1}{2-q}}$, we obtain 
\begin{equation}\label{eq:tends to zero for small q}
|x-\varrho^{(q)}_{h,s}(x,\alpha|x|^{q-1})|=q c_q \alpha|x|^{q-1}.
\end{equation}
Hence, for $q\in [0,1)$, the difference goes to zero as $x$ goes to infinity. 

\begin{figure}
\centering
\subfigure[$q=0.1$]{
\includegraphics[width=.237\textwidth]{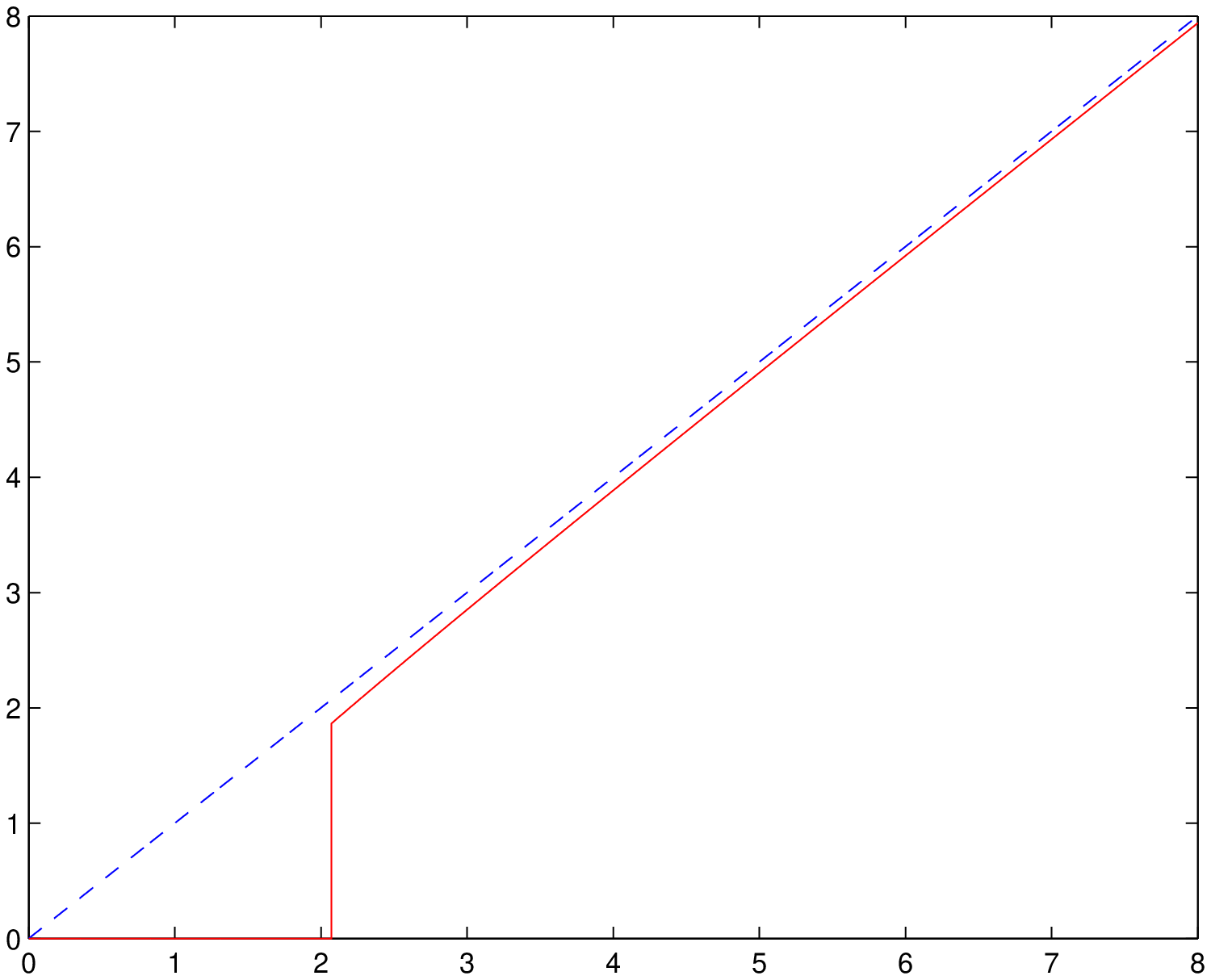}}
\subfigure[$q=0.3$]{
\includegraphics[width=.237\textwidth]{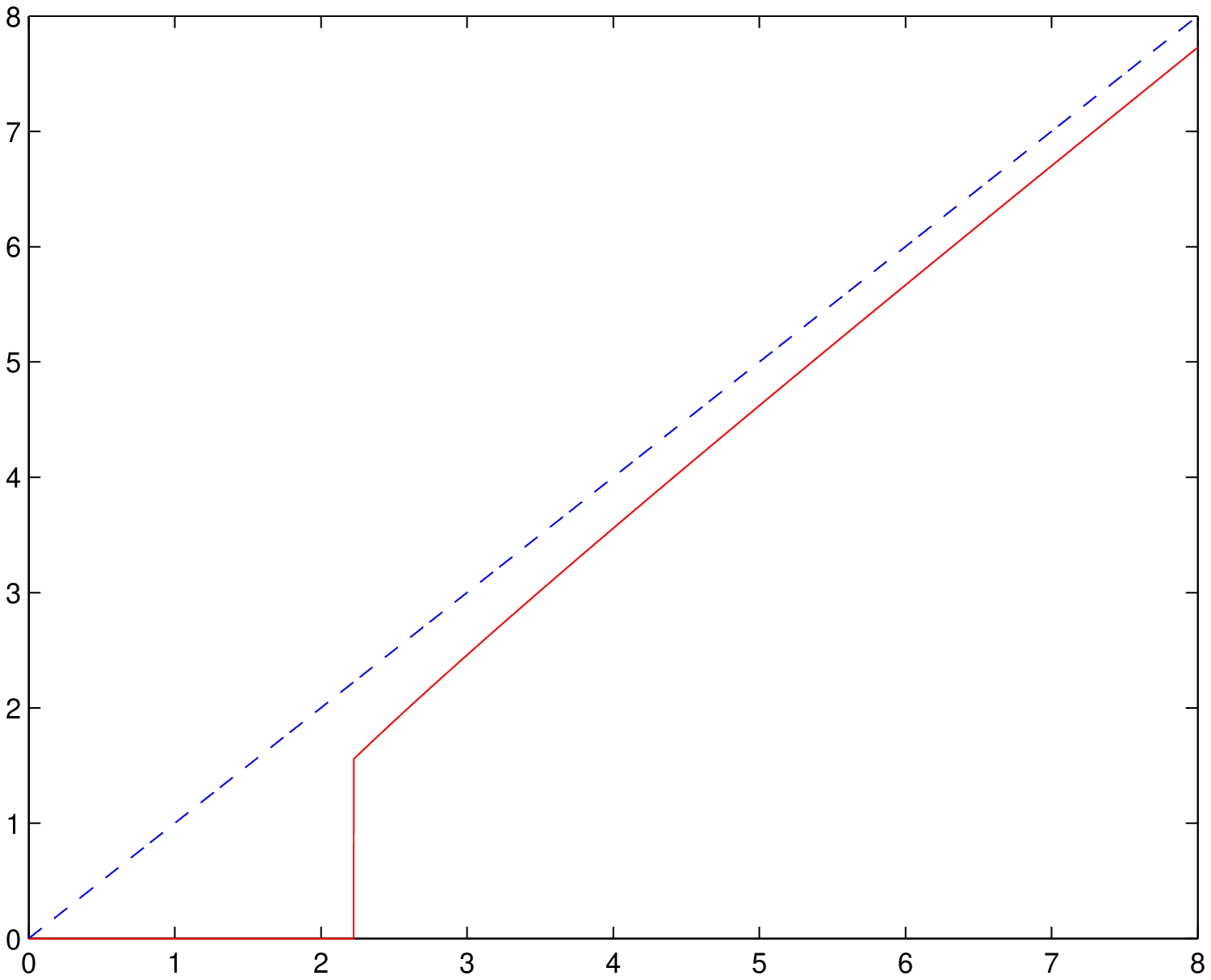}}
\subfigure[$q=1/2$]{
\includegraphics[width=.237\textwidth]{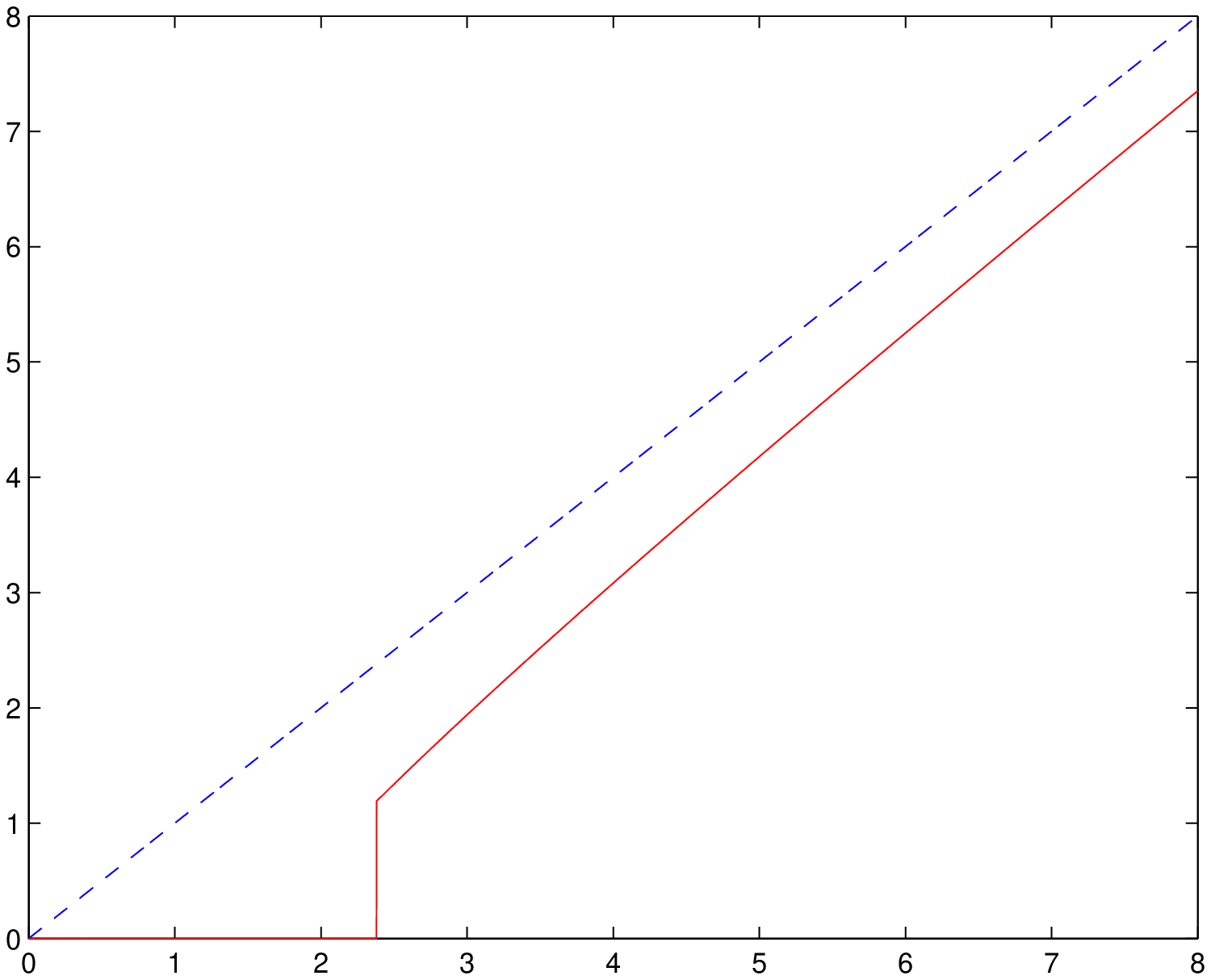}}
\subfigure[$q=0.8$]{
\includegraphics[width=.237\textwidth]{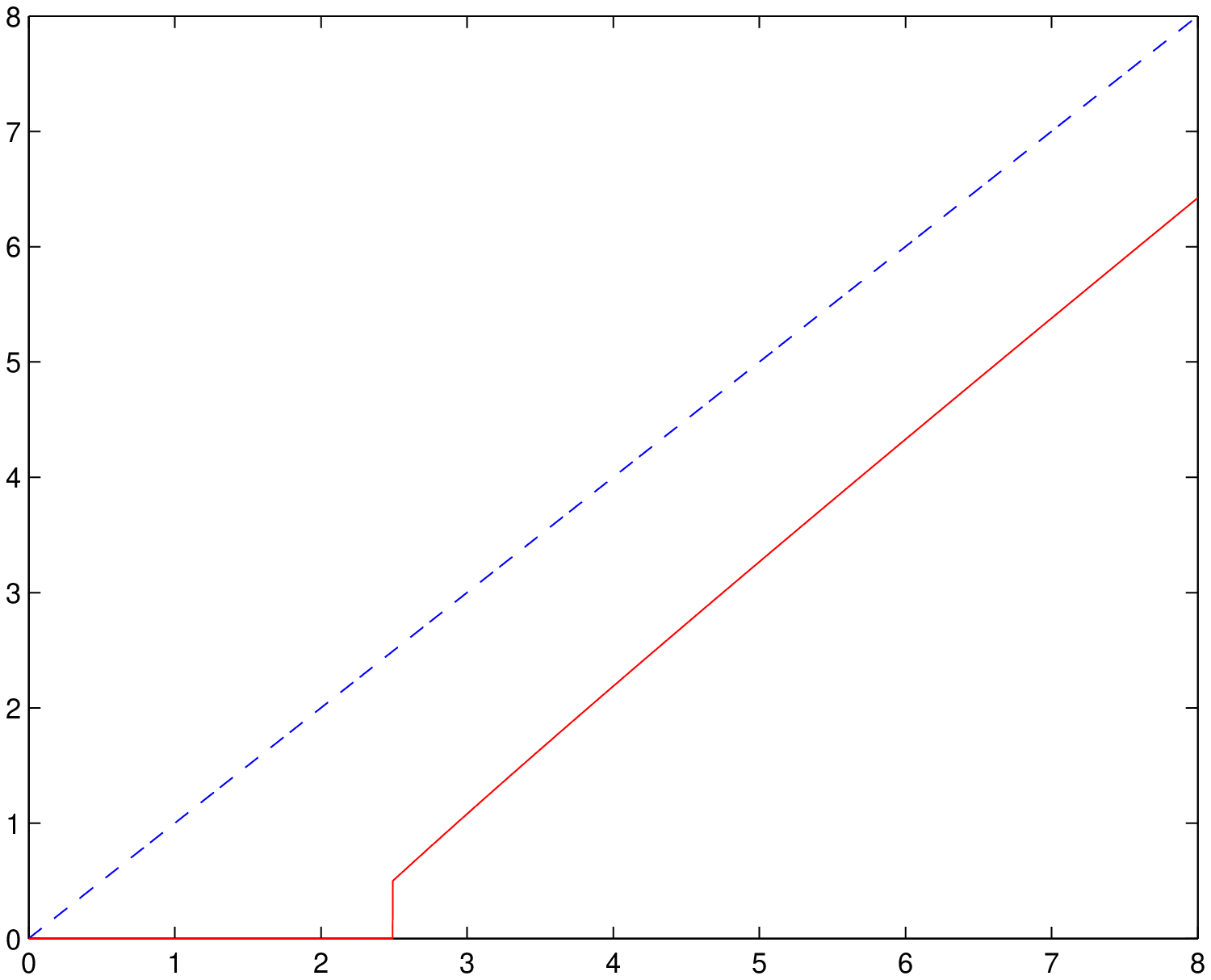}}
\caption{$q$-dependent expression $\varrho^{(q)}_{h,s}(x,\alpha|x|^{q-1})$ for $\alpha=4$. According to \eqref{eq:tends to zero for small q}, the smaller $q$ the faster tends this expression to $x$.}\label{fig new shrinkage rule q dependent}
\end{figure}

The following theorem says that the new rule \eqref{eq:new shrinkage rule} is well adapted to $q\in[0,1]$:
\begin{theorem}\label{theorem varrhi I tau}
The sequence $\varrho^{(q)}_{h,s}(v_n,\alpha_n|v_n|^{q-1})_{n\in\mathcal{N}}$ is an exact minimizer of \eqref{problem equiv} at the endpoints $q=0$, $q=1$. It minimizes \eqref{problem equiv} up to a constant factor in between, and it coincides with the exact minimizer on $\{n\in\mathcal{N} : |v_n| < c_q^{\frac{1}{2-q}}\alpha_n^{\frac{1}{2-q}} \}$.
\end{theorem}
\begin{proof}
Soft-shrinkage $\varrho_s(v_n,\frac{\alpha_n}{2})_{n\in\mathcal{N}}$ is the exact minimizer of \eqref{problem equiv}, for $q=1$, cf.~\cite{variational_denoising}. Note that 
\begin{equation*}
\varrho_s(v_n,\frac{\alpha_n}{2})=\varrho^{(1)}_{h,s}(v_n,\alpha_n),\quad\text{for all $n\in\mathcal{N}$}.
\end{equation*}
The exact minimizer for $q=0$ is hard-shrinkage
$\varrho_h(v_n,\sqrt{\alpha_n})_{n\in\mathcal{N}}$, see \cite{Lorenz:2008ab}, and we have the identity
\begin{equation*}
 \varrho_h(v_n,\sqrt{\alpha_n})=\varrho^{(0)}_{h,s}(v_n,\alpha_n|v_n|^{-1})
\quad\text{for all $n\in\mathcal{N}$}.
\end{equation*}
The shrinkage rule $\varrho^{(q)}_{h,s}$ satisfies \eqref{eq estimate threshold rule zwei} for $\rho=\infty$. Hence due to Proposition \ref{proposition solving equiv}, it minimizes \eqref{problem equiv} up to a constant factor.

We have $\varrho^{(q)}_{h,s}(v_n,\alpha_n|v_n|^{q-1})=0$ iff $|v_n|\leq c_q\alpha_n|v_n|^{q-1}$. Since $|v_n|\leq c_q\alpha_n|v_n|^{q-1}$ is equivalent to $|v_n|^{2-q}\leq c_q\alpha_n$, it is also equivalent to $|v_n|\leq c^{\frac{1}{2-q}}_q\alpha^{\frac{1}{2-q}}_n$,  for $q\in (0,1)$. According to the results in \cite{Lorenz:2008ab}, see also \cite{Antoniadis01regularizationof}, each exact minimizer $(\hat{\omega}_n)_{n\in\mathcal{N}}$ satisfies $\hat{\omega}_n=0$ for $|v_n|< c^{\frac{1}{2-q}}_q\alpha^{\frac{1}{2-q}}_n$. 
\end{proof}
Due to Theorem \ref{theorem varrhi I tau}, the rule $\varrho^{(q)}_{h,s}$ is an adaptation to $q\in[0,1]$. This might also be useful for parameter fitting: While $\alpha=(\alpha_n)_{n\in\mathcal{N}}$ can be fitted to $f$ and $L$ by considering \eqref{eq:curve}, the new family $\varrho^{(q)}_{h,s}$ provides additional flexibility to optimize the choice of $q$ as well. One optimizes $\alpha=\alpha(q)$ as in \eqref{eq:curve}, one may then vary $q\in[0,1]$ and may optimize this sparsity parameter by analyzing the  univariate curve $\alpha(q)$.

\section{Iterative Shrinkage Strategies}\label{section:Landweber}
In the present section, the derived shrinkage strategies in Section \ref{section interpolation} are applied to inverse problems. We slightly change our perspective and consider the problem 
\begin{equation}\label{eq:original}
\arg\min_{g}\big(\| f-\mathcal{T}g\|\big)
\end{equation}
in which the operator $\mathcal{T}$ does not have a bounded pseudo inverse or the norm is extremely big. Such an ill-posed problem needs regularization. During the last decade, regularization with sparsity constraints has attracted significant attention, see, for instance, \cite{Bonesky:2007aa,Bonesky:2009aa,daub_inverse_defrise,Daubechies:2008aa,Engl:1996aa,Fornasier:2007aa,Lorenz:2008ab,Ramlau:2006aa}. One solves
\begin{equation}\label{eq:Fredholm regularized}
\arg\min_g\big(\|f - \mathcal{T}g\|^2+\alpha\phi(g)\big),
\end{equation}
where $\phi$ is a measure of the sparsity of $g$ in some chosen dictionary, and the nonnegative regularization parameter $\alpha$ weights the sparsity term. 

\subsection{Landweber Iteration with Shrinkage}\label{subsection:landweber}
A shrinked Landweber iteration has been developed in \cite{daub_inverse_defrise} to minimize \eqref{eq:Fredholm regularized} for the term $\phi(g)=\sum_k |g_k|^q$, provided that $1\leq q\leq 2$ and $(g_k)_k$ are the coefficients for $g$'s representation in an orthonormal basis. To reduce notation, let us assume that $f$ and $g$ are already discretized and hence are just sequences. For $q=1$, the term $\phi(g)$ enforces sparsity. The  minimization \eqref{eq:Fredholm regularized} then is well-posed, and the iteration given by
\begin{align}
g^0 & = 0, \label{eq:zero starting vector} \\ 
g^{j+1} & = S_\alpha(g^j+\mathcal{T}^*f-\mathcal{T}^*\mathcal{T}g^j), \quad\text{ where $S_\alpha(x)_n=\varrho_s(x_n,\alpha)$}, \label{eq:shrink 1}
\end{align}
converges towards the minimizer of \eqref{eq:Fredholm regularized}, see \cite{daub_inverse_defrise}. Soft-shrinkage occurs in this iterative scheme, because it is the exact minimizer of \eqref{problem equiv}. To address other $1<q\leq 2$ in \eqref{eq:Fredholm regularized} with $\phi(g)=\sum_k |g_k|^q$, we need to apply the shrinkage rule that corresponds to the exact minimizer of \eqref{problem equiv} for this particular $q$. On the other hand, it is shown in \cite{Zarzer:2009aa} that \eqref{eq:Fredholm regularized} with the stronger sparsity requirements $0<p<1$ is still well-posed and a regularization of the original problem \eqref{eq:original}. However, the nonconvexity of $\phi(g)$ in this case makes it difficult to design a numerically attractive algorithm for the actual minimization. The $q$-dependent expression of the shrinkage rule $\varrho^{(q)}$ is not the exact minimizer of \eqref{problem equiv}, but still a minimizer up to a constant factor. Motivated by the results in Section \ref{section interpolation}, we propose to replace soft-shrinkage with the $q$-dependent expression of $\varrho^{(q)}$, i.e., to replace $S_\alpha$ in \eqref{eq:shrink 1} with
\begin{equation*}
\tilde{S}_\alpha(x)_n=\varrho^{(q)}(x_n,\alpha |x_n|^{q-1}).
\end{equation*}
This modified scheme is known to converge for $q=0$ and $q=1$, cf.~\cite{Blumensath1,daub_inverse_defrise}. It is thus reasonable to believe that it also converges for $0<q<1$, which is supported by numerical experiments.


%
%
 
%

\subsection{Analytic Ultracentrifugation}
Sedimentation velocity analytical ultracentrifugation is a method to determine the size distribution of macromolecules in a solute, cf.~\cite{Brown:2008aa,Schuck:2000aa}. The physical model leads to a Fredholm integral equation 
\begin{equation}\label{eq:Fredholm}
f(y)=(\mathcal{T}g) (y)=\int g(x)K(x,y)dx,
\end{equation}
whose kernel $K$ of the integral operator $\mathcal{T}$ represents the sedimentation profile and is only implicitly given through the solution of the Lamm equation, a differential equation discussed in \cite{Lamm:1929aa}. From the experimentally observed signal $f$, one must deduce the particles' or macromolecules' size distribution $g$. However, this is an ill-posed problem and requires regularization. State of the art regularizations for this problem are Tikhonov and maximum entropy regularization in \cite{Cox:1969aa,HANSEN:1992aa,Schuck:2000aa}. Both methods have also been used in combination with Bayesian priors \cite{Brown:2008aa}. 

Partial information about a solute is often available. We consider the case in which we know a-priori that the solute is well separated into molecules of very different sizes. In other words, the seeked size distribution $g$ is sparse, i.e., has only few peaks and is almost zero elsewhere. The sparser the expected distribution the smaller we may want to chose $q$. However, there is a trade off, because we then only minimize up to a constant factor. It seems reasonable to believe that heuristics can be developed to chose a near optimal $q$ for a given experiment.    

Before we apply the proposed iterative scheme to solve the analytical ultracentrifugation problem, we observe that the size distribution $g$ must be nonnegative. We first discretize \eqref{eq:Fredholm} by sampling on a finite grid. By using the nonnegativity as an additional regularization, we modify the application of the shrinkage process $S_\alpha$ in such a way that negative arguments are not shrinked in its original sense, but simply set to zero. It does not introduce any additional discontinuities, because $\varrho(x,\alpha)\rightarrow 0$ as $0\leq x\rightarrow 0$. This procedure enforces a nonnegative limit.

In our numerical experiments, we consider different values of $0<q<1$ and compare the results in terms of how much sparsity we obtain while only introducing a relatively small residual $\| f-\mathcal{T}g\|$. We finally compare these findings to maximum entropy regularization that was used in \cite{Brown:2008aa,Schuck:2000aa} to solve the ill-posed problem of analytical ultracentrifugation. We aim to verify that our proposed scheme can provide sparser solutions with sharper peaks, higher resolution and smaller residual.

\subsection{Numerical Results}

Maximum entropy regularization 
\begin{equation*}
\arg\min_{g}\big(\| f-\mathcal{T}g\|^2_{\ell_2} + \beta\sum_n g_n \ln(g_n) \big)
\end{equation*}
is the state of the art tool to solve \eqref{eq:Fredholm} for the analytical ultracentrifugation, cf.~\cite{Brown:2008aa,Schuck:2000aa}. It has been implemented in the softwaretool SEDFIT \cite{Schuck:2000aa}, that we use as a reference. SEDFIT uses f-statistics to choose $\beta$.

\noindent\textbf{Sharper spikes, fewer nonzero entries, and higher resolution:}\\
Our scheme is applied to a highly pure IgG antibody solute. Due to the purity, the ``correct'' solution to the underlying Fredholm integral equation must be highly sparse. We use $100$ measurements on an equidistant grid to solve the discrete analogue of the integral equation \eqref{eq:Fredholm}. The residual $\| f-\mathcal{T}g_S\|_{\ell_2}$ of the SEDFIT solution $g_S$ is $0.7878$. Although the solution seems sparse, cf.~the red graph in Figure \ref{figure:lg100}, it has many small entries and the spikes are relatively wide.  

 The choices $0<q<1$ promote sparsity and, for sufficiently small $\alpha>0$, our proposed scheme leads to smaller residuals, sharper spikes, and few small entries, cf.~Figure \ref{figure:lg100}. The SEDFIT solution is nonzero between $42$ and $90$. The solution to our proposed scheme has peaks at $45$, $55$, $66$, and $89$, and vanishes in between. To verify that these peaks reflect the antibody solute (i.e., the peaks are real), we compute the maximum entropy solution for $1000$ measurements on an equidistant grid, cf.~Figure \ref{figure:IgG1000}. This SEDFIT solution at this higher resolution has peaks around $45$, $55$, $66$, and $89$. Thanks to the sparsity promoting feature of our scheme, we ``see'' these peaks already with the much broader resolution of only $100$ measurements. We observe that starting the iteration with the SEDFIT solution rather than the zero vector in \eqref{eq:zero starting vector} still leads to the same solution which indicates an intrinsic stability of the proposed scheme. 
 
These results suggest that our proposed scheme has great potential when samples can be assumed to be highly pure. The method then leads to sharper spikes, fewer nonzero entries, and higher resolution.





\begin{figure}[]
\centering
\includegraphics[width=0.67\textwidth]{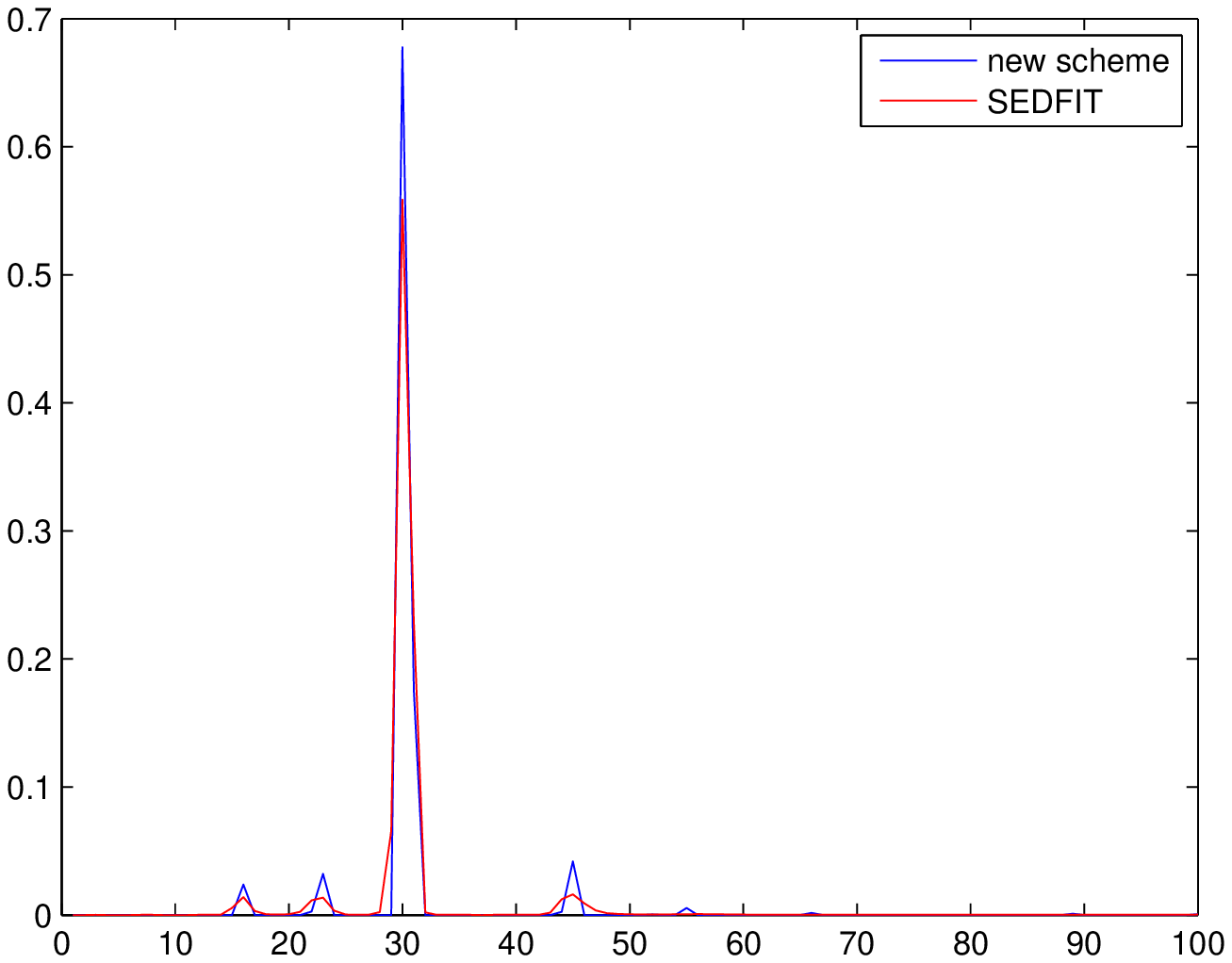}
\includegraphics[width=0.32\textwidth]{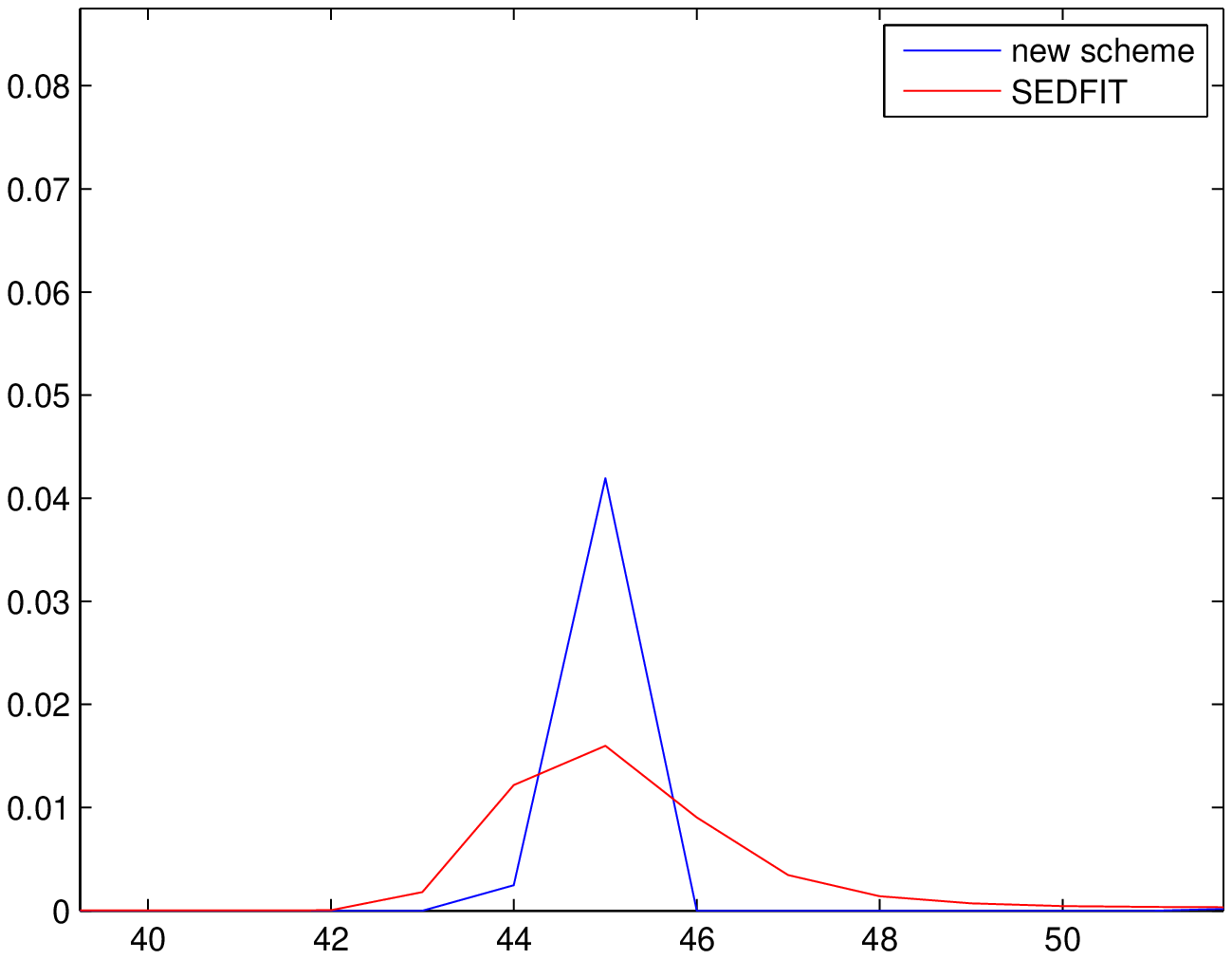}

\includegraphics[width=0.32\textwidth]{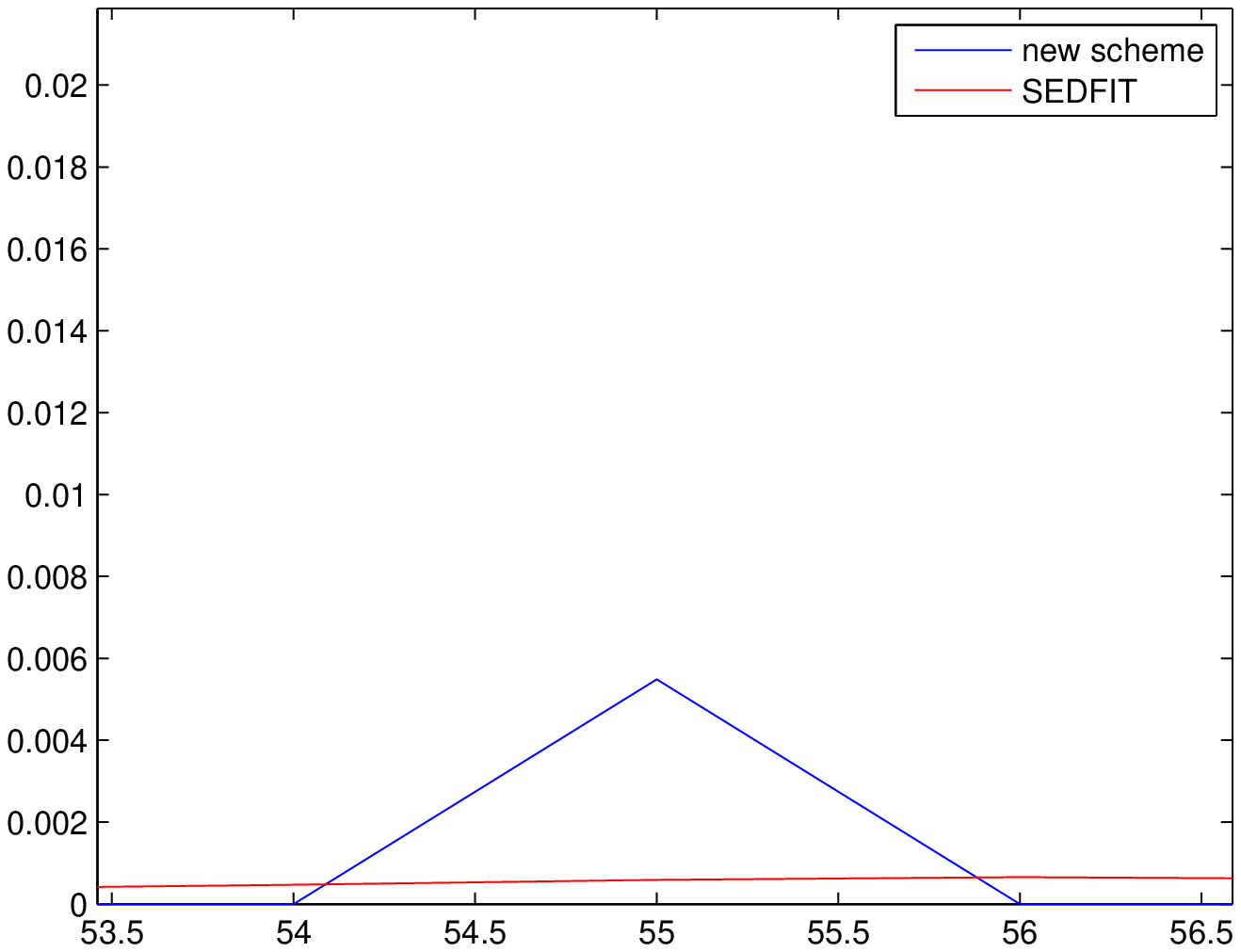}
\includegraphics[width=0.32\textwidth]{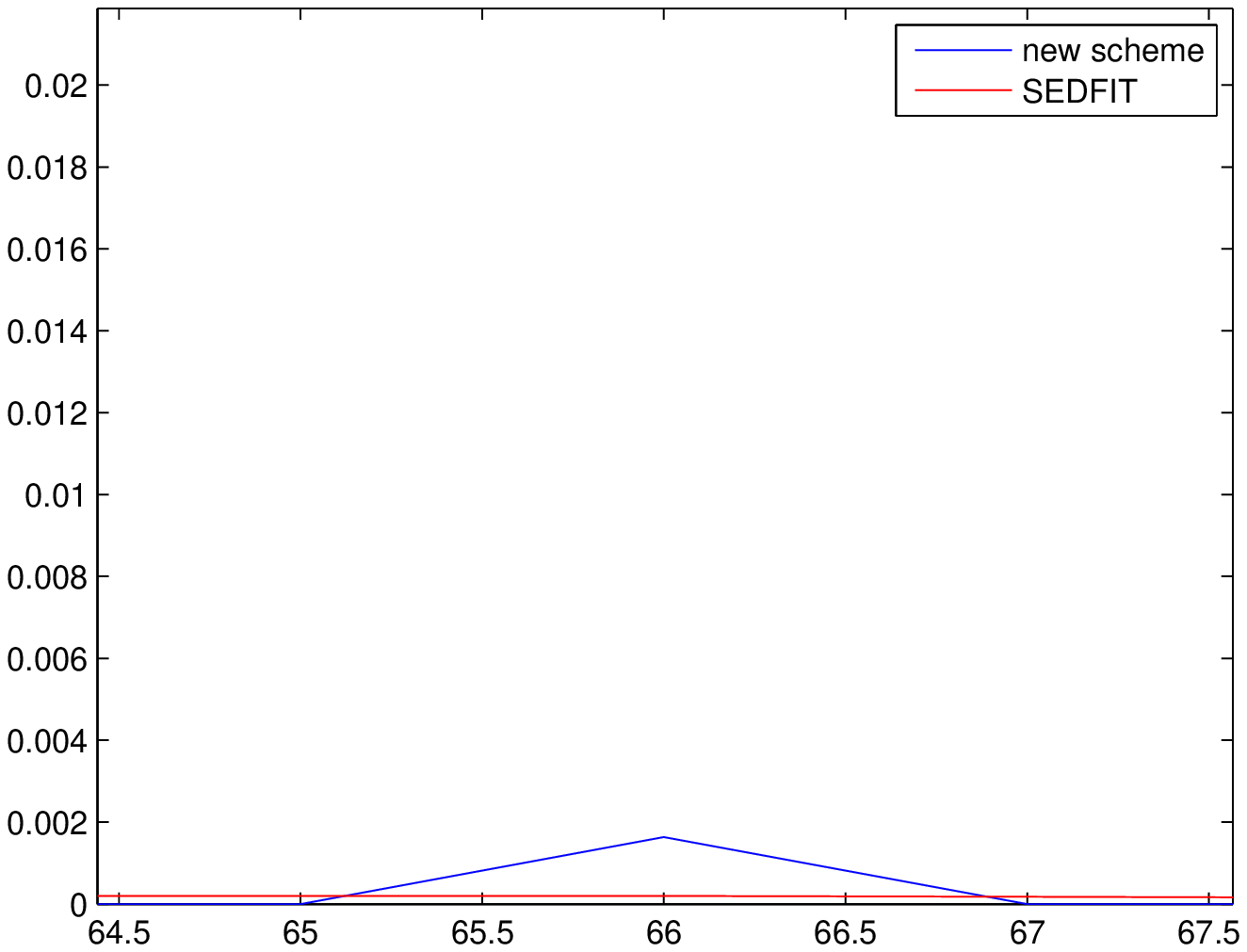}
\includegraphics[width=0.32\textwidth]{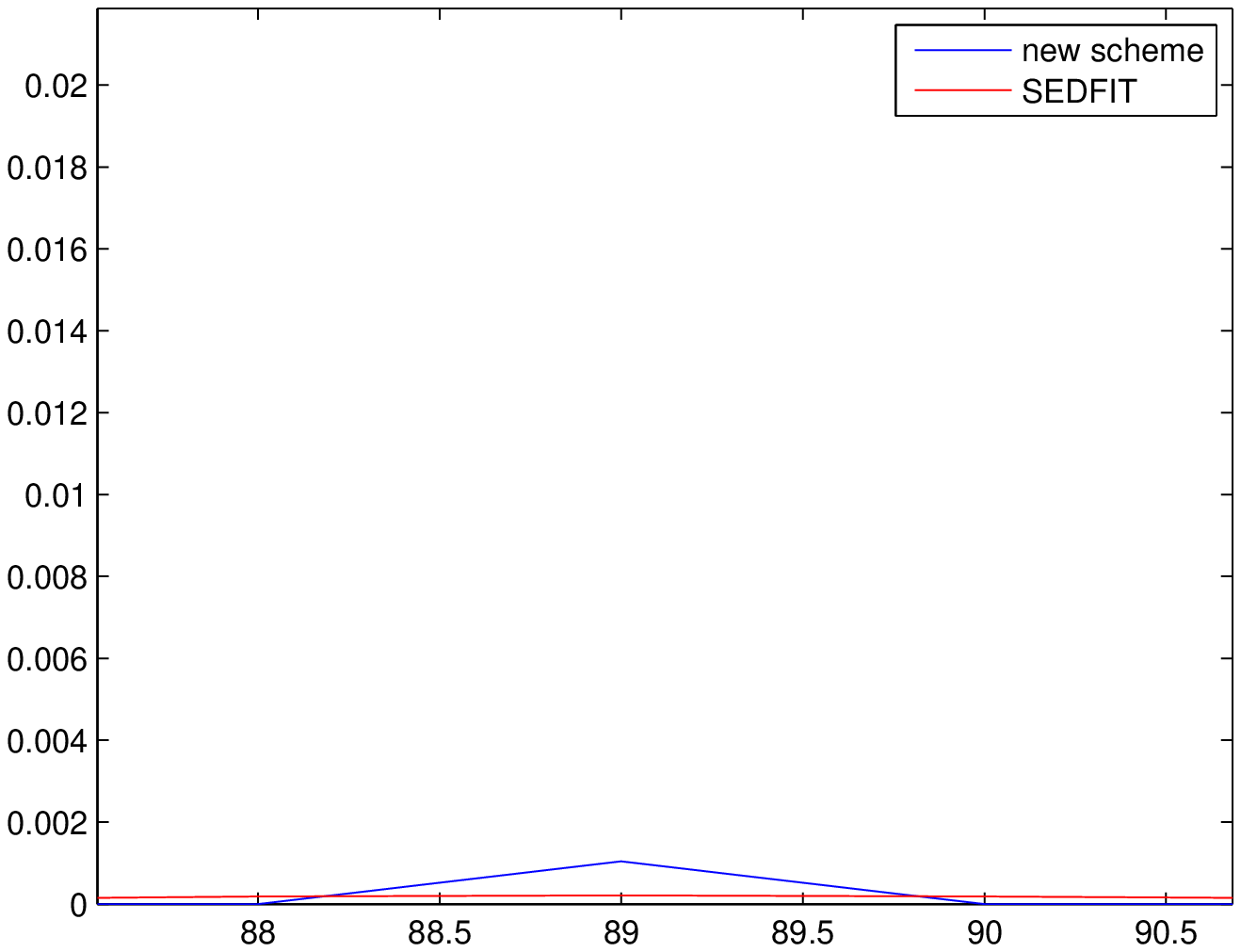}

\caption{Comparison between SEDFIT and our proposed scheme with $q = 0.3$: we obtain sharper peaks, especially around $16$ and $23$. While SEDFIT is nonzero between $42$ and $90$, the solution to our scheme is zero except for peaks at $45$, $55$, $66$, and $89$.}\label{figure:lg100}
\end{figure} 
\begin{figure}[]
\centering
\includegraphics[width=0.32\textwidth]{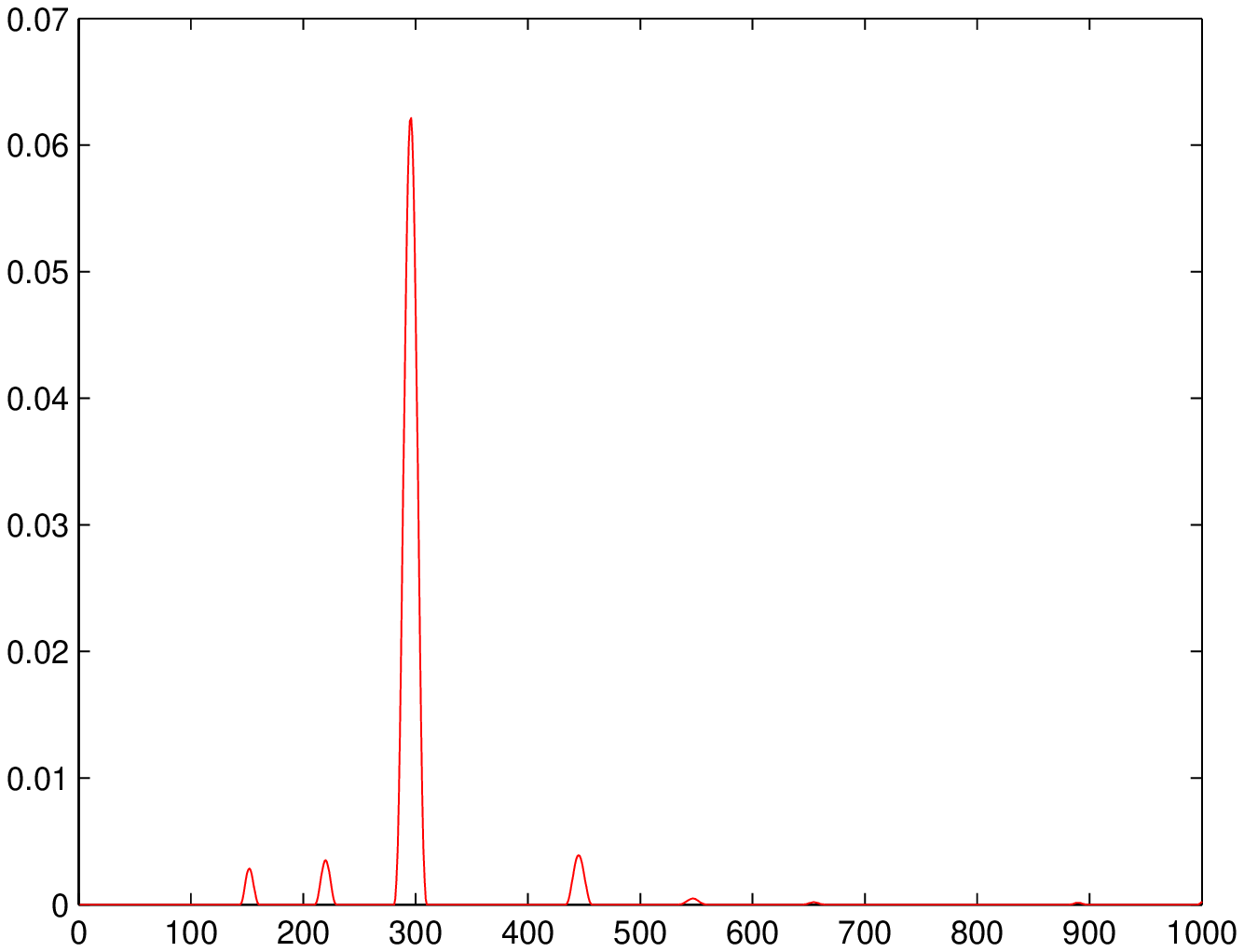}
\includegraphics[width=0.32\textwidth]{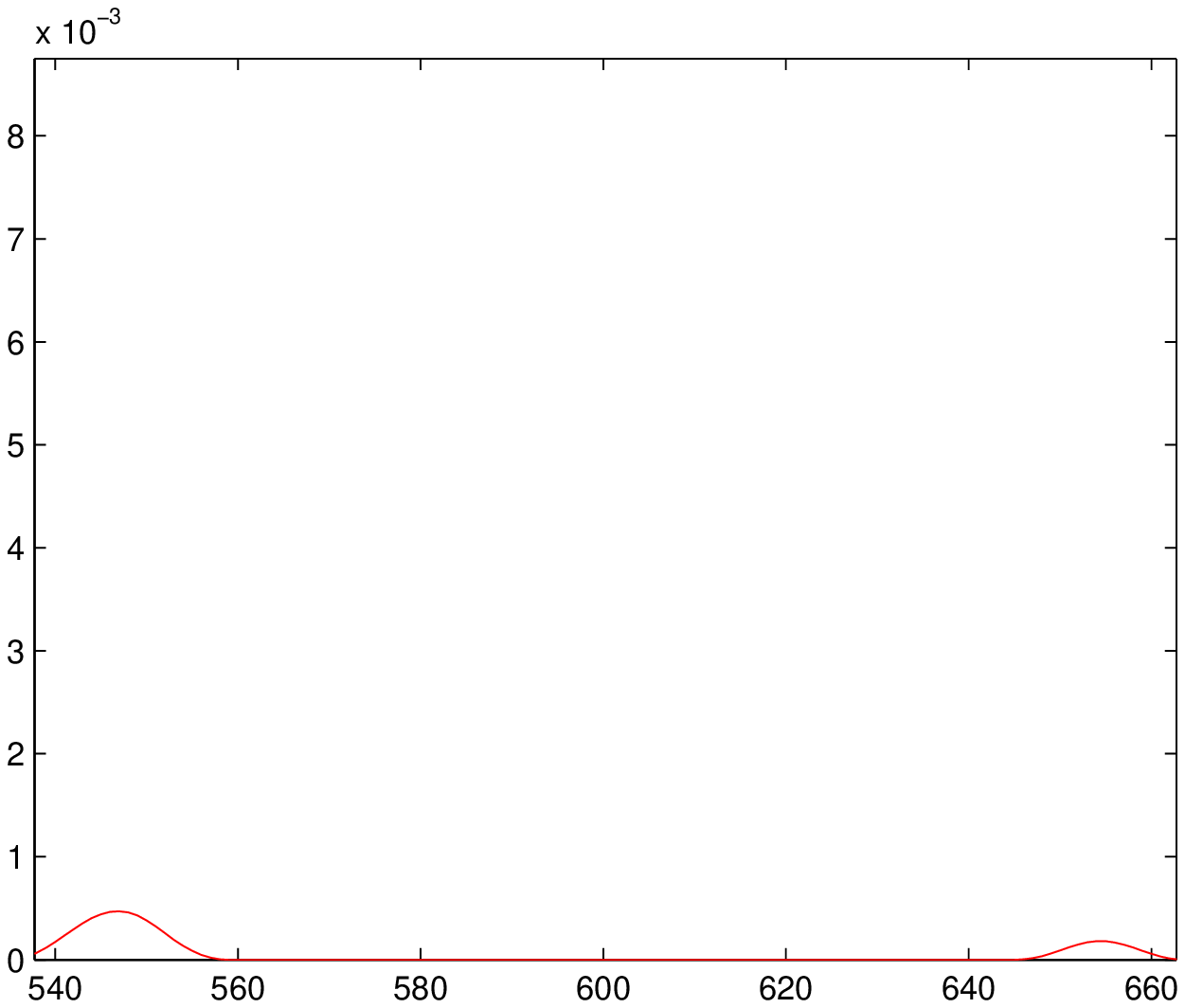}
\includegraphics[width=0.32\textwidth]{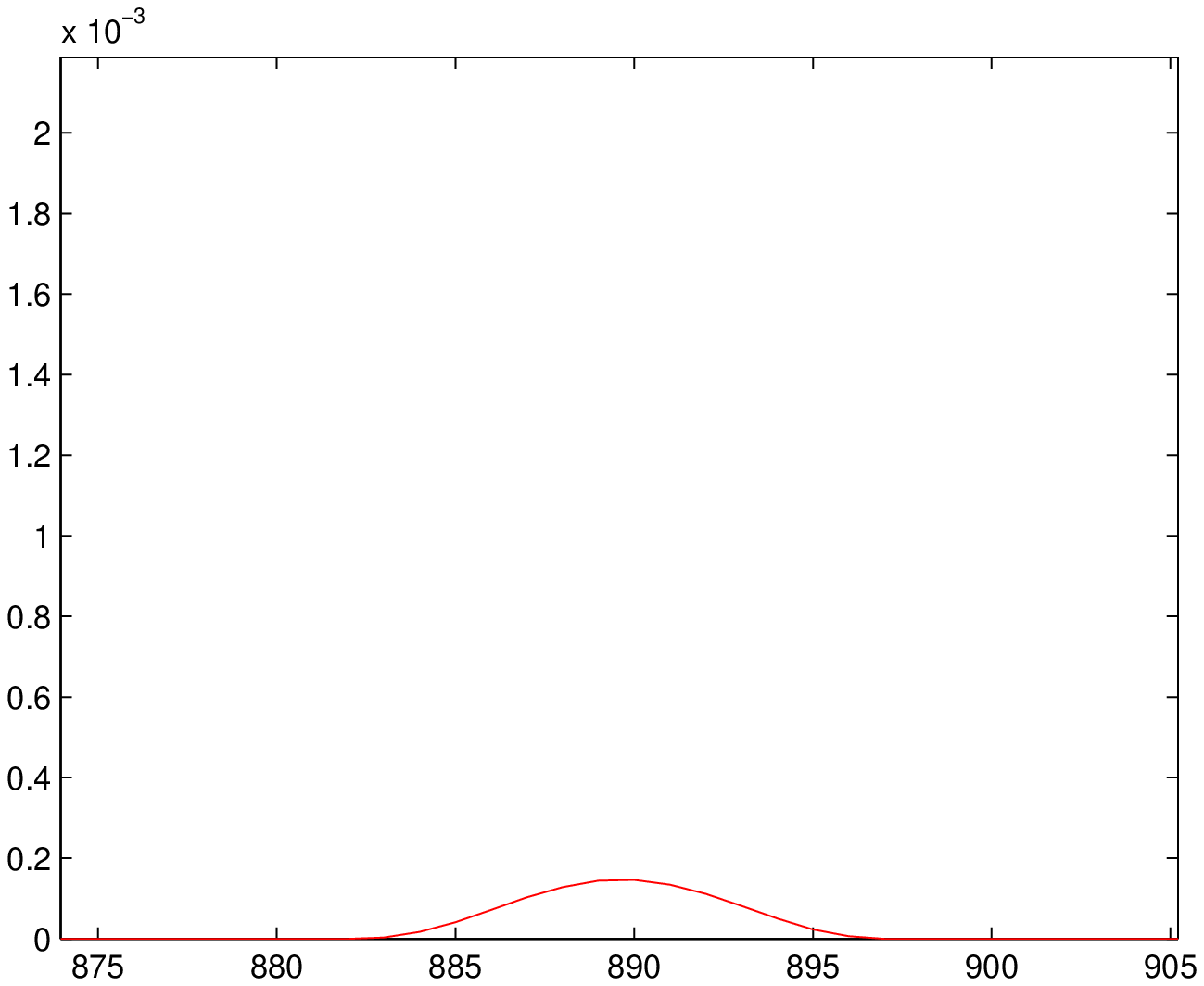}

\caption{SEDFIT solution (maximum entropy regularized) for $1000$ data points. There are peaks at $550$, $660$, and $890$ (to compare with $100$ data points, x-axes needs to be divided by $10$ and y-axes needs to be multiplied by $10$). Thus the peaks of our proposed method at $55$, $66$, and $89$ are real and have the correct amplitude. They therefore reflect the antibody solute while using only a tenth of the data.}\label{figure:IgG1000}
\end{figure} 







\section{Conclusion}\label{section:conclusions}
We have addressed variational problems with $\ell_q$-constraints for $q\in(0,1)$. In case that computation time is crucial as it is in any real-time and on-line application, there are no sufficiently fast algorithms to solve them. By considering minimization up to a constant factor, we have overcome this limitation. We avoid costly iterative schemes and derive closed formulas for such minimizers. This approach provides a tool which makes problems for $q<1$ more feasible than until now. If exact solutions are required, those minimizers can initialize iterative schemes to speed up their convergence and to find an accurate local minimum.  

We have then modified the Landweber iteration with shrinkage applied at each iteration step in \cite{daub_inverse_defrise} by replacing the shrinkage rule with $\varrho^{(q)}_{h,s}(v_n,\alpha_n|v_n|^{q-1})$ to cover $q\in(0,1)$ as well. The proposed scheme has been used to solve the ill-posed problem of analytic ultracentrifugation. The results have been compared to the standard regularization for the analytical ultracentrifugation introduced in \cite{Brown:2008aa,Schuck:2000aa}. We have verified that our proposed scheme can provide sparser solutions with sharper peaks, higher resolution and smaller residual. Thus, the scheme provides a useful add-on to the standard maximum entropy regularization. 

It is known though that iterative schemes of the type presented in Section \eqref{subsection:landweber} converge relatively slowly and thus the computation time of our scheme is orders higher than those in \cite{Brown:2008aa,Schuck:2000aa}. To present a competitive approach that can be incorporated into online applications such as the software-package SEDFIT/SEDPHAT \cite{Schuck:2000aa}, the method still needs a major tune up to derive a faster convergence, cf.~ \cite{Daubechies:2008aa} for possible directions.


For further theoretical foundation, it remains to find general conditions on $L$ and on the bi-frame such that $\widetilde{F}^*L^\#LF$ is bounded on $\ell_q^{(\alpha_n)}$ and to compute the difference between $\varrho^{(q)}_{h,s}$ and the exact minimizer of \eqref{problem equiv}. It also remains to precisely determine the arising constants. For the analytical ultracentrifugation, the brute force discretization by means of sampling must still be replaced with a proper discretization scheme involving suitable ansatz functions and smoothness spaces. It also remains to verify that the Landweber iteration with $q$-dependent shrinkage converges towards a minimizer up to a certain error. We plan to address these topics in a forthcoming paper.

\providecommand{\bysame}{\leavevmode\hbox to3em{\hrulefill}\thinspace}
\providecommand{\MR}{\relax\ifhmode\unskip\space\fi MR }
\providecommand{\MRhref}[2]{%
  \href{http://www.ams.org/mathscinet-getitem?mr=#1}{#2}
}
\providecommand{\href}[2]{#2}

\end{document}